\documentclass[12pt,a4paper,reqno]{amsart}

\usepackage[english]{babel}
\usepackage{amssymb,latexsym,amsfonts,amsthm,upref,amsmath}
\usepackage[foot]{amsaddr} 
\usepackage[margin=1in]{geometry} 
\usepackage[dvipsnames,x11names]{xcolor}
 \definecolor{myblue}{HTML}{003399}
\usepackage{hyperref}
\hypersetup{colorlinks,citecolor=myblue,filecolor=black,linkcolor=myblue,urlcolor=myblue}
\usepackage{enumerate}  
\usepackage{tikz}
\usepackage{float}
\usepackage{cleveref}
\usepackage{mathtools}
\usepackage{cite}
\usepackage{filecontents}
\makeatletter
\newcommand{\leqnomode}{\tagsleft@true}
\newcommand{\reqnomode}{\tagsleft@false}

\makeatother
\newtheorem*{thm*}{Theorem}
\newtheorem*{lem*}{Lemma}
\newtheoremstyle{prim}{}{}{\normalfont}{}{\bfseries}{.}{ }{}
\newtheoremstyle{stil}{}{}{\slshape}{}{\bfseries}{.}{ }{}
\theoremstyle{stil}
\newtheorem{thm}{Theorem}[section]
\newtheoremstyle{defi}{}{}{}{}{\bfseries}{.}{ }{}
\theoremstyle{defi}
\newtheorem{defn}[thm]{Definition}
\theoremstyle{defi}
\newtheorem{rem}[thm]{Remark}
\theoremstyle{stil}
\newtheorem*{mthm*}{Main Theorem}
\newtheorem*{kor*}{Corollary}
\newtheorem{pro}[thm]{Proposition}
\theoremstyle{stil}
\newtheorem{lem}[thm]{Lemma}
\theoremstyle{stil}
\newtheorem{kor}[thm]{Corollary}
\theoremstyle{prim}

\newenvironment{prf}{\noindent \textit{Proof.}}{\null\hfill$\qed$\hskip
2mm\vskip 2mm}

\newcommand{\wht}{\widehat}
\newcommand{\wvr}{\overline}
\newcommand{\wndr}{\underline}

\newcommand{\modd}{ \,{\rm mod\,\,}}

\newcommand{\fand}{\quad\text{and}\quad}
\newcommand{\Fand}{\qquad\text{and}\qquad}
\newcommand{\non}{\nonumber}
\newcommand{\beq}{\begin{equation}}
\newcommand{\eeq}{\end{equation}}

\newcommand{\ot}{\otimes}
\newcommand{\ts}{\hspace{1pt}}
\newcommand{\ndo}{\mathop{\mathrm{End}}}
\newcommand{\om}{\mathop{\mathrm{Hom}}}
\newcommand{\R}{\wvr{R}}

\newcommand{\RRR}{\wndr{R}}

\newcommand{\RRRR}{\wndr{\wvr{R}}}
\newcommand{\uh}{{\rm U}_h (\widehat{\mathfrak{gl}}_N)_c}
\newcommand{\uq}{{\rm U}_q (\widehat{\mathfrak{gl}}_N)}

\newcommand{\yhg}{{\rm Y}_{N,h}^{\text{tw}} }

\newcommand{\dyho}{{\rm DY}_{N,h,c}^{tw}  }

\newcommand{\dyhocrit}{{\rm DY}_{N,h,crit}^{tw}  }
\newcommand{\dyhot}{\widetilde{{\rm DY}}_{N,h,c}^{tw}  }
\newcommand{\dyhocritt}{\widetilde{{\rm DY}}_{N,h,crit}^{tw} }
\newcommand{\dyhoz}{{\rm DY}_{N,h,0}^{tw}  }

\newcommand{\Mcc}{\mathcal{M}^{c} }
\newcommand{\Mcz}{\mathcal{M}^{0} }
\newcommand{\Mccrit}{\mathcal{M}^{crit} }

\newcommand{\vac}{ \mathrm{\boldsymbol{1}}}
\newcommand{\CC}{\mathbb{C}}
\newcommand{\ZZ}{\mathbb{Z}}
\newcommand{\Sc}{\mathcal{S}}
\newcommand{\Ec}{\mathcal{E}}
\newcommand{\Rv}{\check{R}}
\newcommand{\tr}{ {\rm tr}}

\newcommand{\Vc}{\mathcal{V}^{2c}(\mathfrak{gl}_N)}
\newcommand{\Vcz}{\mathcal{V}^{0}(\mathfrak{gl}_N)}
\newcommand{\Vcdva}{\mathcal{V}^{c}(\mathfrak{gl}_N)}
\newcommand{\Vccrit}{\mathcal{V}^{crit}(\mathfrak{gl}_N)}

\newcommand{\bc}{\mathcal{B}}
\newcommand{\cdotrl}{\mathop{\hspace{-2pt}\underset{\text{RL}}{\cdot}\hspace{-2pt}}}
\newcommand{\cdotlr}{\mathop{\hspace{-2pt}\underset{\text{LR}}{\cdot}\hspace{-2pt}}}

\makeatletter
\def\smalloverbrace#1{\mathop{\vbox{\m@th\ialign{##\crcr\noalign{\kern3\p@}%
  \tiny\downbracefill\crcr\noalign{\kern3\p@\nointerlineskip}%
  $\hfil\displaystyle{#1}\hfil$\crcr}}}\limits}
\makeatother

\makeatletter
\def\smallunderbrace#1{\mathop{\vtop{\m@th\ialign{##\crcr
   $\hfil\displaystyle{#1}\hfil$\crcr
   \noalign{\kern3\p@\nointerlineskip}%
   \tiny\upbracefill\crcr\noalign{\kern3\p@}}}}\limits}
\makeatother

\makeatletter
\def\author@andify{%
  \nxandlist {\unskip ,\penalty-1 \space\ignorespaces}%
    {\unskip {} \@@and~}%
    {\unskip \penalty-2 \space \@@and~}%
}
\makeatother

\begin{document}

\title{Invariants of the extended twisted  $h$-Yangian}

\author{Lucia Bagnoli}
\address[L. Bagnoli]{Dipartimento di Matematica, Sapienza Universit\`{a} di Roma, P.le Aldo Moro 5, 00185 Rome, Italy}
\email{lucia.bagnoli@uniroma1.it}

\author{Slaven Ko\v{z}i\'{c}}
\address[S. Ko\v{z}i\'{c}]{University of Zagreb Faculty of Science, Department of Mathematics, Bijeni\v{c}ka cesta 30, 10000 Zagreb, Croatia}
\email{slaven.kozic@math.hr}

\author{Jian Zhang}
\address[J. Zhang]{School of Mathematics and Statistics, Central China Normal University,  Wuhan, Hubei 430079, China}
\email{jzhang@ccnu.edu.cn}

\begin{abstract}
We investigate the extended twisted $h$-Yangian ${\rm Y}_{N,h}^{\text{tw}}$, a certain algebra
which admits both the  orthogonal and the symplectic $h$-Yangian as its quotients.
We show that ${\rm Y}_{N,h}^{\text{tw}}$ is naturally equipped with the structure of restricted module  for a certain algebra ${\rm DY}_{N,h,c}^{tw}$, which resembles the quantum double, as well as   with the structure of
$\phi$-coordinated quasi module for the Etingof--Kazhdan quantum affine vertex algebra associated  with the trigonometric $R$-matrix of type $A$.
Finally, we demonstrate how the elements of the quantum Feigin--Frenkel center give rise to  explicit formulas for  
central elements of ${\rm DY}_{N,h,c}^{tw}$ and
 invariants of ${\rm Y}_{N,h}^{\text{tw}}$ at the critical level, as well as to commutative families in the  orthogonal and  symplectic $h$-Yangians.
\end{abstract}

 \maketitle

\allowdisplaybreaks
 
\section{Introduction}\label{intro}
\setcounter{equation}{0}
\numberwithin{equation}{section}

Twisted $q$-Yangians of orthogonal and symplectic type are two families of  coideal subalgebras of the type $A$ quantum affine algebra, which were   introduced by  Molev,   Ragoucy and Sorba \cite{MRS}.
They possess properties similar to those of their rational counterparts, the twisted Yangians \cite{MNO,O}, which are  coideal subalgebras of the type $A$   Yangian. In particular, the concept of Sklyanin determinant,
named  after   Sklyanin, who introduced
 the new type of   determinant  for a certain class of reflection algebras  \cite{Skly},  plays a key role in understanding the structure of their center \cite{MRS,JZ,M1}.
More recently, Lu \cite{Lu} proved that the twisted $q$-Yangians are  isomorphic to the affine  $\imath$quantum groups, in current presentation  \cite{LuWa},  associated to the symmetric pair of type  AI.

In this paper, we examine a connection between Etingof--Kazhdan's quantum affine vertex algebra $\Vcdva$ 
\cite{EK}, associated with the type $A$ trigonometric $R$-matrix, and the twisted $h$-Yangians, i.e. twisted $q$-Yangians    defined over the commutative ring $\CC[[h]]$. To do so, we introduce   
 the extended twisted $h$-Yangian ${\rm Y}_{N,h}^{\text{tw}}$, a certain algebra
which admits both the  orthogonal and the symplectic $h$-Yangian as its quotients.
The aforementioned connection is established by using Li's notion of $\phi$-coordinated quasi module \cite{Liphi}, while the main challenge in its construction is finding suitable annihilation operators on ${\rm Y}_{N,h}^{\text{tw}}$, so that they yield the corresponding quasi module map.

An analogous problem, in the setting of twisted Yangians and the type $A$ rational $R$-matrix, was studied by the second author and Serti\'{c} \cite{KS}.
Its solution relied on the realization of  the corresponding creation and annihilation operators in terms of the type $A$ centrally extended double Yangian \cite{I}. 
In contrast,  the annihilation operators in the trigonometric case, which is considered in this paper,  do not  seem to possess such a  realization in terms of the type $A$ quantum affine algebra.
Though this may seem unusual from the vertex algebraic perspective,  the similar phenomenon occurs in the Etingof--Kazhdan construction of quantum vertex algebra $\Vcdva$ associated with the trigonometric $R$-matrix.
To address this issue, we introduce a  new algebra ${\rm DY}_{N,h,c}^{tw}$, where $c\in\CC$, along with a certain wide class of ${\rm DY}_{N,h,c}^{tw}$-modules, the restricted modules.
Roughly speaking, this algebra contains both the creation and annihilation operators required for our construction,   and its defining relations are given in the form of three reflection equations.

The main results of the paper are Theorems \ref{thm_25} and \ref{thm_33}. The former defines a structure of restricted ${\rm DY}_{N,h,c}^{tw}$-module on the
extended $h$-twisted Yangian ${\rm Y}_{N,h}^{\text{tw}}$, and the latter states that every restricted ${\rm DY}_{N,h,c}^{tw}$-module is naturally  equipped with the structure of 
$\phi$-coordinated quasi $\Vc$-module.
As an application, we derive explicit formulas for the families of invariants in ${\rm Y}_{N,h}^{\text{tw}}$, with respect to the action of ${\rm DY}_{N,h,-N/2}^{tw}$, and for the families of central elements in a suitable completion of the  algebra ${\rm DY}_{N,h,-N/2}^{tw}$. Both families are   parametrized by Young diagrams and   naturally give  rise to commutative families in the  orthogonal and  symplectic $h$-Yangians.
Our construction relies on 
the fusion procedure for the Hecke algebra \cite{C,IMO,N}    and
 the explicit formulas \cite{BJK,KM} for the elements of the   center of $\Vcdva$ at the critical level $c=-N$.

\section{Preliminaries}\label{section_01}

In this section, we recall the trigonometric $R$-matrix of   type $A$  and the $R$-matrix realization of the quantum affine algebra  of   type $A$;  see     \cite{FRT,J,PS,RS} for more information. Next, following the exposition in \cite[Sect. 4]{BJK}, we recall   Etingof--Kazhdan's construction \cite[Thm. 2.3]{EK} of the quantum   vertex algebra associated to the aforementioned $R$-matrix. 

\subsection{Quantum affine algebra \texorpdfstring{$\uh$}{Uh(glN)c}}\label{section_0101}
Let $N\geqslant 2$.
Due to  \cite{FR}, there exists a unique formal  power series  
\beq\label{efodze00}
f_q(z)=1+\sum_{r=1}^\infty f_{q,r} \frac{z^r}{(1-z)^{r}} \in \CC(q)[[z]] 
\eeq  
  such that all $f_{q,r}(q-1)^{-r}$ are regular at $q=1$, which satisfies the identity
$$
f_q(zq^{2N})=f_q(z) (1-zq^2)(1-zq^{2N-2}) (1-z)^{-1}(1-zq^{2N})^{-1}.
$$
By setting $q=e^{h/2} $ in \eqref{efodze00}, we get a   formal power series in $ \CC[[h,z]],$
$$
f(z)\coloneqq f_{e^{h/2}}(z)=1+\sum_{r=1}^\infty f_{r} \frac{z^r}{(1-z)^{r}} ,\qquad\text{where}\qquad
f_r=f_{q,r}\big|_{q=e^{h/2}}.
$$

 Let $V=\ndo\CC^N\ot\ndo\CC^N.$ 
Introduce  the    $R$-matrix   $\R (z )\in V  [[ h]][z],$    
\begin{align}
\R (z ) =&\,\left(1-ze^{-h}\right)\sum_{i=1}^N e_{ii}\ot e_{ii} 
+ e^{-h/2}\left(1-z\right)\sum_{\substack{i,j=1\\i\neq j}}^N e_{ii}\ot e_{jj} \non\\
&+  \left(1-e^{-h }\right)z   \sum_{\substack{i,j=1\\i> j}}^N e_{ij}\ot e_{ji}
+  \left(1-e^{-h } \right) \sum_{\substack{i,j=1\\i< j}}^N e_{ij}\ot e_{ji},\label{rbarmatrix}
\end{align}
where $e_{ij}\in\ndo\CC^N$ are the matrix units. 
Next,  define the normalized $R$-matrix by
\beq\label{erofzed}
R(z)=  \frac{f(z)   \ts \R(z)}{1-ze^{-h}}.
\eeq
Clearly, the $R$-matrix \eqref{erofzed} belongs to $V[[z,h]].$ On the other hand, 
due to the form of the term  $f(z)$,  it can be also  regarded as an element  of $V(z)[[h]].$ It possesses the property
$R(z)^{t_1\ts t_2} = R_{21}(z)$,
where $t_k$ with $k=1,2$ stands for the matrix transposition $e_{ij}\mapsto e_{ji}$ applied on the $k$-th tensor factor of $V$. Moreover, it satisfies
the   {\em quantum Yang--Baxter equation},  
\beq\label{trig_ybe_12}
R_{12}(z_1)\ts R_{13}(z_1 z_2)\ts R_{23}(z_2)
 =   R_{23}(z_2)\ts R_{13}(z_1 z_2)\ts R_{12}(z_1) 
\eeq
and the {\em crossing symmetry identities,}
\beq\label{csym}
R(ze^{ Nh})^{t_1} \ts D_1\ts  ( R(z)^{-1})^{t_1}=D_1\Fand (R(z)^{-1})^{t_2} \ts D_2 \ts R(z e^{ Nh})^{t_2} = D_2,
\eeq
where
$D_1=D\ot 1,  D_2=1\ot D \in V$
and $D$ is the diagonal $N\times N$ matrix
\beq\label{diagonal}
D=\mathop{diag}\left(e^{ (N-1)h/2 },e^{ (N-3)h/2},\ldots ,e^{- (N-1)h/2} \right) .
\eeq 
Denote by
``$\cdotrl$'' (resp.  ``$\cdotlr$'')
the standard multiplication in $(\ndo\CC^N)^{\text{op}} \ot\ndo\CC^N$
(resp. $ \ndo\CC^N  \ot(\ndo\CC^N)^{\text{op}}$), where $A^{\text{op}}$ denotes the opposite algebra of a given associative algebra $A$. The crossing symmetry properties \eqref{csym} imply that the $R$-matrix $R(z)$ is invertible with respect to such multiplications, e.g.,   the first identity implies
$$
R(z) \cdotrl R(z)^\sim = R(z) \cdotlr R(z)^\sim =1\quad\text{for}\quad R(z)^\sim = 
D_1^{-1}\ts R(ze^{-Nh})\ts D_1.
$$
Throughout the paper, we shall often denote by the symbol ``$\sim$''   the inverse of the $R$-matrix  with respect to such multiplications.
Finally, we remark that the $R$-matrix
$$
\wht{R}(z)=\frac{ \R(z)}{1-ze^{-h}} = \frac{R(z)}{f(z)} 
$$
satisfies the {\em unitarity condition},
$$
\wht{R}(z)\ts \wht{R}_{21}(1/z) =\wht{R}_{21}(1/z)\ts \wht{R}(z)=1.
$$

Let $c\in\CC .$ The {\em quantum affine algebra $\uh$ at the level $c$} is an $h$-adically complete associative algebra over the ring $\CC[[h]]$ generated by the elements $l_{ij}^{\pm(r)},$ where $i,j=1,\ldots ,N $ and $r=0,1,\ldots.$ Its defining relations are given in terms of the matrices
$$
L^\pm(z)=\sum_{i,j=1}^N e_{ij} \ot l_{ij}^\pm (z),\quad\text{where}\quad
l_{ij}^\pm (z)=\delta_{ij}\mp h\sum_{r=0}^\infty l_{ij}^{\pm(r)} z^{\pm r},
$$
as follows:
\begin{gather}
l_{ij}^{-(0)}=l_{ji}^{+(0)}=0\text{ for all }1\leqslant i<j\leqslant N,\non\\
R(x/y)\ts L_1^\pm(x)\ts L_2^\pm(y)=L_2^\pm(y)\ts L_1^\pm(x)\ts R(x/y),\label{dere2l_2}\\
R(xe^{-hc/2}/y)\ts L_1^+(x)\ts L_2^-(y)=L_2^-(y)\ts L_1^+(x)\ts R(xe^{ hc/2}/y).\label{dere2l_3}
\end{gather} 
Throughout the paper, we often use the standard tensor notation, where the subscript indicates the copy in the corresponding tensor product algebra, e.g., 
$$
L_r^\pm(x)=\sum_{i,j=1}^N 1^{\ot (r-1)}\ot e_{ij} \ot 1^{\ot(n-r)}\ot l_{ij}^\pm (z)\in(\ndo\CC^N)^{\ot n}\ot\uh[[z^{\pm 1}]].
$$
In particular, in the defining relations \eqref{dere2l_2} and \eqref{dere2l_3} , we have $n=2$ and $r=1,2$.

\subsection{Quantum affine vertex algebra \texorpdfstring{$\Vcdva$}{Vc(glN)}}\label{section_0102}
Denote by   $\CC_*(u)$     the localization of the ring of formal Taylor series $\CC[[u]]$ at $\CC[u]\setminus\left\{0\right\}$. 
The unique embedding 
$\CC_*(u)\hookrightarrow\CC((u))   $ 
 naturally extends to the embedding
$\iota_u \colon \CC_*(u)[[h]]\hookrightarrow\CC((u)) [[h]].$
Let us regard the expression in \eqref{erofzed} as a formal Taylor   series in $h.$
By applying the substitution $z=e^u$ and then the map $\iota_u ,$ one obtains an element $R^\prime(u)$ of
$\ndo\CC^N \ot\ndo\CC^N  ((u))[[h]]$;  see \cite{KM} for more details. 
Furthermore, by \cite[Prop. 1.2]{EK4} and \cite[Prop. 2.1]{KM}, there exists a unique $\psi\in 1+h\CC[[h]]$ such that the $R$-matrix 
\beq\label{erofu}
R(e^u)\coloneqq \psi \ts R^\prime(u)\in \ndo\CC^N \ot\ndo\CC^N ((u))[[h]]
\eeq
satisfies the {\em unitarity condition},
\beq\label{unitarity}
R_{12}(e^u) \ts R_{21}(e^{-u}) =R_{21}(e^{-u})\ts R_{12}(e^u)=  1,
\eeq
and  
the {\em crossing symmetry identities},
$$
R(e^{u+Nh})^{t_1} \ts D_1\ts  ( R(e^u)^{-1})^{t_1}=D_1\fand (R(e^u)^{-1})^{t_2} \ts D_2 \ts R( e^{u+Nh})^{t_2} = D_2.
$$
In addition, the $R$-matrix \eqref{erofu}   
satisfies the   {\em quantum Yang--Baxter equation},
$$
R_{12}(e^u)\ts  R_{13}(e^{u+v}) \ts R_{23}(e^v)=R_{23}(e^v)\ts R_{13}(e^{u+v}) \ts R_{12}(e^u).
$$

We now follow  \cite{EK3} to introduce a certain version of the    quantized universal enveloping    algebra; see also \cite{FRT,RS}.
It is defined as a topologically free associative algebra $\textrm{U}(R)$ over the ring $\CC[[h]]$ generated by the elements $t_{ij}^{(-r)}$, where  $i,j=1,\ldots ,N$ and $r=1,2,\ldots .$ Define the matrix $T^+ (u) $   by
$$
T^+ (u) =\sum_{i,j=1}^N e_{ij}\ot t^+_{ij}  (u),
\quad\text{where}\quad
 t^+_{ij} (u)=\delta_{ij}-h\sum_{r=1}^{\infty}t_{ij}^{(-r)}u^{r-1} .
$$
The defining relations for $\textrm{U}(R)$ are given by
\beq\label{rtt}
R(e^{u-v})\ts T^+_{1} (u)\ts T^+_2  (v)=  T^+_2  (v)\ts T^+_{1} (u)\ts R(e^{u-v}).
\eeq

Let $\vac$ be the unit in   $\textrm{U}(R).$   We shall  consider   $T^+ (u)$  as  an operator series  on $\textrm{U}(R)$,
where its action is given by the algebra multiplication. 
Let us recall  \cite[Lemma 2.1]{EK}. 

\begin{lem}\label{lemma21}
For any $c\in\CC$ there exists a unique  operator series
$$T^- (u)\in\ndo\CC^N \ot \om( \textrm{U}(R),\textrm{U}(R) [[u^{\pm 1}]] )$$
which satisfies
$$
  T^-_{0} (u)\ts T_{1}^+ (v_1)\ldots T_{n}^+ (v_n)\vac=
R_{01}^{-1}\ldots R_{0n}^{-1}
\ts T_{1}^+ (v_1)\ldots T_{n }^+ (v_n) \vac\ts  R_{0n}^- \ldots R_{01}^-
$$
for all $n\geqslant 1$, where
$ 
R_{0i}^{-1} = R_{0i} (e^{u-v_i+hc/2})^{-1}$
 and
$R_{0i}^-=R_{0i}  (e^{u-v_i-hc/2})
$.
\end{lem}

For any $n\geqslant 1$ and the variables $u=(u_1,\ldots ,u_n)$ define
$$
T_{[n]}^{\pm} (u )=T_1^\pm( u_1)\ldots T_n^\pm ( u_n)
\fand
T_{[n]}^{\pm} (u|z)=T_1^\pm(z+u_1)\ldots T_n^\pm (z+u_n).
$$ 
Finally, we recall the Etingof--Kazhdan  construction \cite[Thm. 2.3]{EK}.

\begin{thm}\label{EK:qva}
For any $c\in\CC,$ there exists a unique   quantum vertex algebra  structure\footnote{The theorem of Etingof and Kazhdan \cite[Thm. 2.3]{EK} provides an explicit expression for the underlying braiding map as well. However, we do not use it in this paper, so we omitted it from our exposition.} on the $\CC[[h]]$-module $\Vcdva\coloneqq \textrm{U}(R)$
  such that the vacuum vector is $\vac\in \textrm{U}(R) $	and
	the vertex operator map $Y(\cdot ,z)$ is defined by
\beq\label{EK_vo}
Y\big(T^+_{[n]}  (u)\vac,z\big)=T^+_{[n]}  (u|z)\ts T_{[n]}^- (u|z+hc/2)^{-1}. 
\eeq
\end{thm}

\section{Extended twisted \texorpdfstring{$h$}{h}-Yangian}\label{section_02}

In this section, we study a certain  algebra  which is motivated by the realization of   the orthogonal and symplectic twisted $q$-Yangians  as  coideal subalgebras  of the quantum affine algebra $\uq,$ given by Molev,   Ragoucy and Sorba \cite{MRS}. We refer to it as the {\em extended twisted $h$-Yangian $\yhg$}. It  is defined as 
an associative algebra over the ring $\CC[[h]]$ generated by the elements $b_{ij}^{ (r)},$ where $i,j=1,\ldots ,N $ and $r=0,1,\ldots.$ Its defining relations are expressed in terms of the matrix
$$
B (z)=\sum_{i,j=1}^N e_{ij} \ot b_{ij}  (z),\quad\text{where}\quad
b_{ij}  (z)=\delta_{ij}- h\sum_{r=0}^\infty b_{ij}^{ (r)} z^{  r},
$$
as  the   reflection equation 
\begin{gather}
R(x/y)\ts B_1 (x)\left(R_{21}(xy)^{-1}\right)^{t_1}  B_2 (y)
=
B_2 (y)\left(R_{21}(xy)^{-1}\right)^{t_1}B_1 (x)\ts R(x/y).\label{re}
\end{gather}
 Throughout the paper, we assume that  the algebra $\yhg$ is $h$-adically completed. 

The next proposition can be verified by a direct computation which closely follows the proof of   \cite[Prop. 4.20]{JZ}. 

\begin{pro}\label{pro_21}
\begin{enumerate}
\item
The
assignment  
\beq\label{map1_10}
B (z)\mapsto L^+(ze^{hc/2}) \ts L^-(1/z)^t 
\eeq 
defines an algebra homomorphism $\yhg\to\uh.$ 
\item
Suppose that $N $ is even and let 
$$G=e^{-h/2}\sum_{r=1}^n e_{2r-1\ts 2r} - \sum_{r=1}^n e_{2r\ts 2r-1}\quad\text{for}\quad n=N/2 .$$ The
assignment  
\beq\label{map2_10}
B (z)\mapsto L^+(ze^{hc/2})\ts G \ts L^-(1/z)^t 
\eeq
defines an algebra homomorphism $\yhg\to\uh.$ 
\end{enumerate}
\end{pro}

\begin{rem}\label{tw_10}
By the realization \cite[Sect. 3.3]{MRS} of the twisted $h$-Yangians, the image of the map \eqref{map1_10} (resp. \eqref{map2_10}) is isomorphic to the ($h$-adically completed) orthogonal (resp. symplectic) twisted $h$-Yangian defined over the ring $\CC[[h]]$.
\end{rem}

The reflection equation \eqref{re} can be generalized as follows. Let $n$ and $m$ be positive integers, $x=(x_1,\ldots ,x_n) $ and $y=(y_1,\ldots ,y_m)$ the families of variables and   $a\in\CC.$  Introduce the $R$-matrix products with coefficients in $(\ndo\CC^N)^{\ot n} \ot (\ndo\CC^N)^{\ot m} $,  
\begin{align}
& R_{nm}^{12}(xe^{ah}/y)=
\prod_{r=1,\dots,n }^{\longrightarrow} \prod_{s=n+1,\dots,n+m }^{\longleftarrow}
R_{rs}(x_r e^{ah}/y_{s-n}),\label{rnm1}\\
& \RRR_{nm}^{21}(xye^{ah})^\prime=
\prod_{r=1,\dots,n }^{\longrightarrow} \prod_{s=n+1,\dots,n+m }^{\longrightarrow}
\left(R_{sr}(x_r y_{s-n}e^{ah})^{-1}\right)^{t_r}.\label{rnm2}
\end{align}
Note that the superscript $1$ (resp. $2$) denotes the tensor factors $1,\ldots ,n$ (resp. $n+1,\ldots ,n+m$).
Next, define the formal powers series with coefficients in  $(\ndo\CC^N)^{\ot n} \ot \yhg,$
\beq\label{bnovi_09}
B_{[n]} (x) =\prod_{i=1,\ldots ,n }^{\longrightarrow}
\left(B_i (x_i)\left( R_{i+1\ts i}(x_i x_{i+1})^{-1}\right)^{t_i}\cdots
\left( R_{n  i}(x_i x_{n})^{-1}\right)^{t_i}\right).
\eeq
The next proposition can be proved by a direct computation which relies on the reflection equation \eqref{re}.
\begin{pro}
For any $m,n\geqslant 1$, we have
$$
R_{nm}^{12}(x/y)\ts B_{[n]}^{  13}(x)\ts \RRR_{nm}^{21}(xy)^\prime \ts  B_{[m]}^{  23}(y)
=
 B_{[m]}^{  23}(y)\ts \RRR_{nm}^{21}(xy)^\prime\ts B_{[n]}^{  13}(x)\ts R_{nm}^{12}(x/y),
$$
where the superscripts $1,2,3$ indicate the tensor factors as follows:
$$ 
\smalloverbrace{(\ndo\CC^N)^{\ot n}}^{1}
\ot
\smalloverbrace{(\ndo\CC^N)^{\ot m}}^{2}
\ot
\smalloverbrace{\yhg}^{3}.
$$
\end{pro}

Let $c\in\CC$.
In addition to the  extended twisted $h$-Yangian, we   introduce a certain algebra $\dyho $.
It is defined as 
an $h$-adically completed associative algebra over the ring $\CC[[h]]$ generated by the elements $b_{ij}^{ + (r)} $ and $b_{ij}^{ * (s)} ,$  where $i,j=1,\ldots ,N ,$   $r=0,1,\ldots $ and $s\in\ZZ .$ Its defining relations are given in terms of the matrices
\begin{align*}
B^+(z)=\sum_{i,j=1}^N e_{ij} \ot b_{ij}^+ (z),\quad&\text{where}\quad
b_{ij}^+ (z)=\delta_{ij}- h\sum_{r=0}^\infty b_{ij}^{+ (r)} z^{   r},\\
B^*(z)=\sum_{i,j=1}^N e_{ij} \ot b_{ij}^* (z),\quad&\text{where}\quad
b_{ij}^* (z)=\delta_{ij}+ h\sum_{r=-\infty}^\infty b_{ij}^{* (r)} z^{   r},
\end{align*}
as three reflection equations,
\begin{align}
&R(x/y)\ts B_1^+(x)\left(R_{21}(xy)^{-1}\right)^{t_1}  B_2^+(y)
=
B_2^+(y)\left(R_{21}(xy)^{-1}\right)^{t_1}B_1^+(x)\ts R(x/y),\label{re1}\\
&\R(x/y)\ts B_1^*(x)\ts \R(1/xy)^{t_1}  B_2^*(y)
=
B_2^*(y)\ts \R(1/xy)^{t_1} B_1^*(x)\ts \R(x/y),\label{re2}\\ 
&R_{21}(ye^{ - hc }/x)^{-1}  B_1^*(x) \left(R_{21}(xye^{-hc })^{-1}\right)^{t_1}  B_2^+(y)\non\\
&\qquad\qquad=
B_2^+(y)\left(R_{21}(xye^{ hc })^{-1}\right)^{t_1} B_1^*(x)R_{21}(ye^{  hc}/x)^{-1} .\label{re3} 
\end{align}

\begin{rem}
In view of the $R$-matrix realization of the quantum affine algebra, it might seem peculiar that  the series $b_{ij}^*(z)$ possess infinitely many positive powers of the variable $z$. However, this is a common phenomenon, which occurs for the annihilation operators in the case of the trigonometric $R$-matrix \cite[Sect. 2.1.2]{EK}, which is exactly the role the  series $b_{ij}^*(z)$ take in the construction given by Theorem \ref{thm_33} below.
\end{rem}

\begin{rem}
By comparing the defining relations   \eqref{re}   and \eqref{re1}, we see that the  
assignment $B(z)\mapsto B^+(z)$ defines an algebra homomorphism $\yhg\to\dyho.$
\end{rem}

The next proposition is verified by arguing as in the proof of \cite[Prop. 4.20]{JZ}.

\begin{thm}\label{pro_24}
For any $c\in\CC$ the assignments
$$
B^+(z)\mapsto L^+(z)\ts L^-(e^{-hc/2}/z)^t\Fand
B^*(z)\mapsto L^-(ze^{-hc })\ts L^+(e^{ hc/2 }/z)^t
$$
define an algebra homomorphism
\beq\label{homo}
\dyhoz \to \uh.
\eeq
\end{thm}

\begin{rem}
Note that the homomorphism \eqref{homo} annihilates all elements $b_{ij}^{*(r)}$ with $r>0 $. 
Furthermore, we remark that it is not clear how to extend Theorem \ref{pro_24}  to the case of algebra $\dyho$ with $c\neq 0$.
\end{rem}

 The   relations \eqref{re1}--\eqref{re3} can be generalized as follows. Let   $n$ and $m$  be positive integers  and  $x=(x_1,\ldots ,x_n), $   $y=(y_1,\ldots ,y_m)$ the families of variables. Introduce the $R$-matrix products with coefficients in $(\ndo\CC^N)^{\ot n} \ot (\ndo\CC^N)^{\ot m} $, 
\begin{align}
& R_{nm}^{21}(ye^{ah}/x)^\prime=
\prod_{r=1,\dots,n }^{\longrightarrow} \prod_{s=n+1,\dots,n+m }^{\longleftarrow}
R_{sr}( y_{s-n}e^{ah}/x_r)^{-1},\label{rnm3}\\
& \RRR_{nm}^{12}(e^{ah}/xy) =
\prod_{r=1,\dots,n }^{\longrightarrow} \prod_{s=n+1,\dots,n+m }^{\longrightarrow}
  R_{rs}(e^{ah}/x_r y_{s-n})^{t_r}.\label{rnm4}
\end{align}
We shall extend the notation \eqref{rnm1},  \eqref{rnm2},  \eqref{rnm3} and  \eqref{rnm4}   to the   $R$-matrix $\R(z)$ in an obvious way, e.g., we write (cf. \eqref{rnm1}  and  \eqref{rnm4})
\begin{align*}
& \R_{nm}^{12}(xe^{ah}/y)=
\prod_{r=1,\dots,n }^{\longrightarrow} \prod_{s=n+1,\dots,n+m }^{\longleftarrow}
\R_{rs}(x_r e^{ah} /y_{s-n}),\\
&\RRRR_{nm}^{12}(e^{ah}/xy) =
\prod_{r=1,\dots,n }^{\longrightarrow} \prod_{s=n+1,\dots,n+m }^{\longrightarrow}
 \R_{rs}(e^{ah}/x_r y_{s-n})^{t_r}.
\end{align*}
 Moreover, define  the formal powers series with coefficients in  $(\ndo\CC^N)^{\ot n} \ot \dyho,$
\begin{align}
& B^+_{[n]} (x) =\prod_{i=1,\ldots ,n }^{\longrightarrow}
\left(B^+_i (x_i)\left( R_{i+1\ts i}(x_i x_{i+1})^{-1}\right)^{t_i}\cdots
\left( R_{n  i}(x_i x_{n})^{-1}\right)^{t_i}\right),\label{nbplus}\\
& B^*_{[n]} (x) =\prod_{i=1,\ldots ,n }^{\longrightarrow}
\left(B^*_i (x_i)\ts \R_{ i\ts i+1}(e^{-hc}/x_i x_{i+1})^{t_i}\cdots
 \R_{  i n}(e^{-hc}/x_i x_{n}) ^{t_i}\right).\label{nbminus}
\end{align}
The following proposition can be proved by using  the reflection equations \eqref{re1}--\eqref{re3}.
\begin{pro}
For any $m,n\geqslant 1$, $x=(x_1,\ldots x_n)$ and $y=(y_1,\ldots y_m)$, we have
\begin{align}
&R_{nm}^{12}(x/y)\ts B_{[n]}^{+\ts  13}(x)\ts \RRR_{nm}^{21}(xy)^\prime \ts  B_{[m]}^{ +\ts  23}(y)
=
 B_{[m]}^{+\ts   23}(y)\ts \RRR_{nm}^{21}(xy)^\prime\ts B_{[n]}^{+\ts   13}(x)\ts R_{nm}^{12}(x/y),\label{re_gen1}\\
&\R_{nm}^{12}(x/y)\ts B_{[n]}^{*\ts  13}(x)\ts \RRRR_{nm}^{12}(1/xy) \ts  B_{[m]}^{*\ts  23}(y) =
 B_{[m]}^{*\ts   23}(y)\ts \RRRR_{nm}^{12}(1/xy)\ts B_{[n]}^{*\ts   13}(x)\ts \R_{nm}^{12}(x/y),\label{re_gen2}\\
&R_{nm}^{21}(ye^{- hc }/x)^\prime\ts B_{[n]}^{*\ts  13}(x)\ts \RRR_{nm}^{21}(xye^{-hc })^\prime  B_{[m]}^{ +\ts  23}(y)\non\\
&\qquad\qquad=
 B_{[m]}^{+\ts   23}(y)\ts \RRR_{nm}^{21}(xye^{ hc })^\prime  B_{[n]}^{*\ts   13}(x)\ts R_{nm}^{21}(ye^{ hc }/x)^\prime ,\label{re_gen3} 
\end{align}where the superscripts $1,2,3$ indicate the tensor factors as follows:
$$ 
\smalloverbrace{(\ndo\CC^N)^{\ot n}}^{1}
\ot
\smalloverbrace{(\ndo\CC^N)^{\ot m}}^{2}
\ot
\smalloverbrace{\dyho}^{3}.
$$
\end{pro}

At the end of this section, we introduce a certain completion of the algebra $\dyho$.
Let $I_p$ with $p\geqslant 1$ be the $h$-adically completed two-sided ideal in    $\dyho $ generated by $h^p$ and all elements $b_{ij}^{*(r)}$, where $r\leqslant -p$, such that   $h^n a\in I_p$ implies  $  a\in I_p$ for all $n\geqslant 1$.
Introduce the completed algebra as the inverse limit
$$
\dyhot=\lim_{ \longleftarrow} \, \dyho / I_p.
$$

\section{Restricted \texorpdfstring{$\dyho$}{DYNhctw}-modules}\label{section_03}

Let $W$ be a  topologically free $\CC[[h]]$-module which is also a $\dyho$-module. Then $W$ is said to be   {\em restricted}, if the action $B^*(z)_W$ of the series $B^*(z)$ on $W$ belongs to
$\ndo\CC^N\ot\om(W,W((z))_h)$, where $W((z))_h$ denotes the $h$-adic completion of $W((z))$. 
In this section, we study restricted $\dyho$-modules. Our main results are Theorem \ref{thm_25}, which defines a  structure of restricted $\dyho$-module on $\yhg$, and Theorem \ref{thm_33}, which demonstrates that restricted $\dyho$-modules are naturally equipped with a structure of $\phi$-coordinated quasi $\Vc$-module.
 
In the next theorem, we use the indices $0,1,2$ to denote the tensor factors as follows:
\beq\label{faktori_09}
\smalloverbrace{\ndo\CC^N}^{0}
\ot
\smalloverbrace{(\ndo\CC^N)^{\ot m}}^{1}
\ot
\smalloverbrace{\yhg}^{2}.
\eeq
Moreover, we denote by ``$\cdotrl$''   
the standard multiplication in $(\ndo\CC^N)^{\text{op}} \ot(\ndo\CC^N)^{\ot m}$
and by
$\left(\RRR_{1m}^{10}(xye^{-hc })^\prime\right)^\sim$   the inverse of
$ \RRR_{1m}^{10}(xye^{-hc })^\prime $ with respect to this multiplication, which is uniquely determined by the property
$\left(\RRR_{1m}^{10}(xye^{-hc })^\prime\right)^\sim   \cdotrl  \RRR_{1m}^{10}(xye^{-hc })^\prime  =1$.  

\begin{thm}\label{thm_25}
For any $c\in\CC$ there exists a unique structure of restricted $\dyho$-module on $\yhg$ such that for any $m$ and the variables $x$ and  $y=(y_1,\ldots ,y_m)$ we have
\begin{align}
&B_0^+(x)\ts  B_{[m]}^{  12}(y)   =  B_0 (x)  B_{[m]}^{  12}(y) ,\label{actionbplus}\\
&  B_0^{* }(x)\ts   B_{[m]}^{  12}(y) =
\left(\RRR_{1m}^{10}(xye^{-hc })^\prime\right)^\sim \cdotrl\Big(
 \left(R_{1m}^{10}(ye^{- hc }/x)^\prime\right)^{-1} B_{[m]}^{   12}(y)\Big.\non\\ 
&\qquad \qquad\qquad\qquad\qquad\qquad\qquad\hspace{39pt}\Big.\RRR_{1m}^{10}(xye^{ hc })^\prime  \ts R_{1m}^{10}(ye^{ hc }/x)^\prime \Big).\label{actionbstar}
\end{align}
\end{thm}

\begin{prf}
Let $y=(y_1,\ldots ,y_m)$ and
\beq\label{tmpf004}
R_{[m]} (y)=\prod_{i=1,\ldots ,m-1 }^{\longrightarrow}
\left( \left( R_{i+1\ts i}(y_i y_{i+1})^{-1}\right)^{t_i}\cdots
\left( R_{m i}(y_i y_{m})^{-1}\right)^{t_i}\right).
\eeq
By \eqref{bnovi_09}, 
we have
\beq\label{bnovi09}
 B_{[m]}(y)=R_{[m]} (y)\cdotrl \left(B_1(y_1)\ldots  B_m(y_m)\right),
\eeq
where the symbol ``$\cdotrl$'' means that each factor 
$$
  \left(R_{i+1\ts i}(y_i y_{i+1})^{-1}\right)^{t_i}\cdots
\left( R_{m i}(y_i y_{m})^{-1}\right)^{t_i}
$$
of $R_{[m]} (y)$
acts on $ B_1(y_1)\ldots  B_m(y_m) $ by  the standard multiplication in $(\ndo\CC^N)^{\text{op}}\ot(\ndo\CC^{N})^{\ot(n-i)}$.
Clearly, \eqref{bnovi09} can be written equivalently as
\beq\label{tmpf002}
 R_{[m]} (y)^\sim \cdotrl B_{[m]}(y)=  B_1(y_1)\ldots  B_m(y_m),
\eeq
where
\begin{align}
R_{[m]} (y)^\sim=&\prod_{i=1,\ldots ,m-1 }^{\longleftarrow}
 R_{m i}^{t_i}  \cdots
 R_{i+1\ts i}^{t_i}\qquad\text{for}\qquad R_{ji}^{t_i}=R_{ji}(y_i y_j)^{t_i}.\label{tmpf08}
\end{align}
Therefore,  if there exists a structure of 
$\dyho$-module on $\yhg$, such that the  \eqref{actionbplus} and \eqref{actionbstar}  hold, it is determined uniquely by these formulas.
Moreover, such a module needs to be restricted, due to the form of the $R$-matrix \eqref{erofzed}. Indeed, it suffices to observe that for any $n,a_1,\ldots ,a_m\geqslant 1$, the right-hand side of \eqref{actionbstar} possesses finitely many negative powers of $x$ modulo $h^n y_1^{a_1}\ldots y_m^{a_m}$.

To prove the theorem, it remains to show that \eqref{actionbplus} and \eqref{actionbstar} define $\CC[[h]]$-module maps $B^+(x)$ and $B^*(x)$ on $\yhg$ and, furthermore, that these maps satisfy the reflection equations \eqref{re1}--\eqref{re3}, which define the algebra $\dyho$. Regarding the latter, this is checked by applying the left and right-hand sides of \eqref{re1}--\eqref{re3} on  $B_{[m]}(y)$ and then using  \eqref{actionbplus} and \eqref{actionbstar} to show that   both sides coincide. As for the former, one needs to show that the ideal of  defining relations \eqref{re} for $\yhg$ lies in the kernel  of the maps \eqref{actionbplus} and \eqref{actionbstar}. Due to the same form of \eqref{re} and  \eqref{re1}, this follows easily for the first operator $B^+(x)$.  As for the second operator $B^*(x)$, the proof is more complicated so we present it in detail.

To declutter the notation, we  write
\beq\label{tmpf003}
B_{(m)} (y) = B_1(y_1)\ldots B_m(y_m) .
\eeq
Instead of \eqref{actionbstar}, we shall use the equivalent formula 
\beq\label{tmpf_001}
  B_0^{* }(x)\ts   B_{(m)}^1(y) =
	R_{[m]}^1 (y)^\sim \cdotrl\left(
 X \cdotrl\left(
 Y \ts B_{(m)}^1(y)\ts Z  \ts W \right)\right),
\eeq
where the terms $R_{[m]}^1 (y)^\sim= R_{[m]} (y)^\sim$  and $B_{(m)}^1(y)=B_{(m)} (y)$, defined by \eqref{tmpf004} and \eqref{tmpf003}, act  on the tensor factors denoted by the superscript $1$ in \eqref{faktori_09}. Also, we use the notation
\begin{align}
&X=\left(\RRR_{1m}^{10}(xye^{-hc })^\prime\right)^\sim, &&Y=\left(R_{1m}^{10}(ye^{- hc }/x)^\prime\right)^{-1},\label{xandy}\\
&Z=\RRR_{1m}^{10}(xye^{ hc })^\prime ,&& W=R_{1m}^{10}(ye^{ hc }/x)^\prime ,\label{zandw}
\end{align}
which employs the superscripts as in \eqref{faktori_09}.
Note that the equivalence of expressions \eqref{actionbstar} and \eqref{tmpf_001} is a direct consequence of
\eqref{tmpf002}.

Suppose that $m\geqslant 2$ and let $k=1,\ldots ,m-1$. As before, we consider the family of variables $y=(y_1,\ldots ,y_m)$. Denote by $y^\prime$ the family obtained by swapping $y_k$ and $y_{k+1}$,  
$$
 y^\prime=(y_1,\ldots ,y_{k-1},y_{k+1},y_k,y_{k+2},\ldots ,y_m) .
$$
Next, let 
\beq\label{permutationoperator}
P=\sum_{i,j=1}^N e_{ij}\ot e_{ji}\in\ndo\CC^N \ot\ndo\CC^N
\eeq
be the permutation operator.
Consider the defining relation \eqref{re}. Let
\beq\label{tmpf20}
U=R_{k\ts k+1}(y_k/y_{k+1})\Fand
V=\left(R_{k+1\ts k}(y_k y_{k+1})^{-1}\right)^{t_k}.
\eeq
It is sufficient to prove that the images of
\beq\label{tmpf005}
U
\left(
V\cdotrl
B_{(m)}^1(y)
\right)
\Fand 
\left(V
\cdotlr
\left(P_{k\ts k+1}\ts
B_{(m)}^1(y^\prime)\ts
P_{k\ts k+1}\right) 
\right)
U
\eeq
under \eqref{tmpf_001} coincide. Regarding the notation in \eqref{tmpf005}, note that the symbol ``$\cdotrl$'' (resp. ``$\cdotlr$'')
for the action of $V=\left(R_{k+1\ts k}(y_k y_{k+1})^{-1}\right)^{t_k}$ indicates that its $k$-th (resp. $(k+1)$-th) component  is applied from the right and its $(k+1)$-th (resp. $ k $-th) component from the left. In other words, the action of $V$ on the tensor factors $k$ and $k+1$ is given by the standard multiplication in $(\ndo\CC^N)^{\text{op}}\ot\ndo\CC^N$
(resp. $\ndo\CC^N \ot (\ndo\CC^N)^{\text{op}} $).

Let us   examine the image of the first expression in \eqref{tmpf005}.   By \eqref{tmpf_001},   it  equals
\beq\label{tmpf09}
U
\left(
V\cdotrl
	\left(
	R_{[m]}^1 (y)^\sim \cdotrl\left(
 X \cdotrl\left(
 Y \ts B_{(m)}^1(y)\ts Z  \ts W \right)\right)
\right)
\right).
\eeq
Consider the $R$-matrix product $R_{[m]}^1 (y)^\sim$, which is given by \eqref{tmpf08}. Multiplying it by $V=\left(R_{k+1\ts k}(y_k y_{k+1})^{-1}\right)^{t_k}$,  as in \eqref{tmpf09}, cancels its factor 
$R_{k+1\ts k}^{t_k}$.
Next, multiplying the resulting expression by 
$U=R_{k\ts k+1}(y_k/y_{k+1})$, as in \eqref{tmpf09}, swaps its adjacent factors $R_{k+1 i}^{t_i}$ and $R_{k  i}^{t_i}$  with $i=1,\ldots ,k-1$. Indeed, this follows from the Yang--Baxter equation \eqref{trig_ybe_12}.
Therefore, the expression in \eqref{tmpf09} coincides with
\beq\label{tmpf10}
	R_{[m]}^1 (y)^\circ \cdotrl\left(
 U
\left(X \cdotrl\left(
 Y \ts B_{(m)}^1(y)\ts Z  \ts W \right)\right)
\right),
\eeq
where $R_{[m]}^1 (y)^\circ $ is obtained from $R_{[m]}^1 (y)^\sim $
by removing its factor $R_{k+1\ts k}^{t_k}$ and then swapping the adjacent factors $R_{k+1 i}^{t_i}$ and $R_{k  i}^{t_i}$  for $i=1,\ldots ,k-1$. In other words, $R_{[m]}^1 (y)^\circ $   equals
\beq\label{tmpf14}
\left(\prod_{i=k+1,\ldots ,m }^{\longleftarrow}
 R_{m i}^{t_i}  \ldots 
 R_{i+1\ts i}^{t_i}\right)
 R_{m k}^{t_k}  \ldots 
 R_{k+2\ts k}^{t_k} 
\left(\prod_{i=1,\ldots ,k-1 }^{\longleftarrow}
 R_{m i}^{t_i}  \ldots R_{k+2 i}^{t_i} \ts 
R_{k  i}^{t_i} \ts R_{k+1 i}^{t_i} \ts
R_{k-1 i}^{t_i} \cdots
 R_{i+1\ts i}^{t_i}\right),
\eeq
where  $R_{ji}^{t_i}=R_{ji}(y_i y_j)^{t_i}$.

Consider the terms $X$ and $Y$, given by \eqref{xandy}. We have
\begin{align*}
& X=\prod_{s=1,\ldots ,m  }^{\longleftarrow}
R_{s0}(xy_s e^{-hc })^{t_0}=R_{m0}(xy_m e^{-hc })^{t_0}\ldots R_{10}(xy_1 e^{-hc })^{t_0},\\
&Y=\prod_{s=1,\ldots ,m  }^{\longrightarrow}
R_{s0}( y_s e^{- hc }/x) 
=R_{10}( y_1 e^{- hc }/x) \ldots R_{m0}( y_m e^{- hc }/x) 
.
\end{align*}
Denote by $X^\circ$ (resp.  $Y^\circ$) the $R$-matrix product obtained from $X$ (resp. $Y$) by swapping its adjacent factors 
$R_{k+1\ts 0}(xy_{k+1} e^{-hc })^{t_0}$ and $R_{k0}(xy_k e^{-hc })^{t_0}$
(resp. $R_{k0}( y_k e^{- hc }/x)$ and $R_{k+1\ts 0}( y_{k+1} e^{- hc }/x)$).
 By the Yang--Baxter equation \eqref{trig_ybe_12}, we have
$$
U\ts X =X^\circ \ts  U\Fand
U\ts Y =Y^\circ \ts  U.
$$
Thus, we conclude that \eqref{tmpf10} is equal to
\beq\label{tmpf11}
	R_{[m]}^1 (y)^\circ \cdotrl 
\left(X^\circ \cdotrl\left(
 Y^\circ \ts U\ts B_{(m)}^1(y)\ts Z  \ts W \right) 
\right).
\eeq

Next, by the reflection equation \eqref{re}, we see that
$$
U\ts B_{(m)}^1(y)=\left(P_{k\ts k+1}\ts
B_{(m)}^1(y^\prime)\ts
P_{k\ts k+1}\right) U,
$$
which implies that \eqref{tmpf11} is equal to
\beq\label{tmpf12}
	R_{[m]}^1 (y)^\circ \cdotrl 
\left(X^\circ \cdotrl\left(
 Y^\circ   \left(P_{k\ts k+1}\ts
B_{(m)}^1(y^\prime)\ts
P_{k\ts k+1}\right)  U\ts Z  \ts W \right) 
\right).
\eeq

Consider the terms $Z$ and $W$, given by \eqref{zandw}.
We have
\begin{align*}
& Z=\prod_{s=1,\ldots ,m  }^{\longrightarrow}
\left(R_{s0}(xy_s e^{ hc })^{-1}\right)^{t_0}=
\left(R_{10}(xy_1 e^{ hc })^{-1}\right)^{t_0}\ldots 
\left(R_{m0}(xy_m e^{ hc })^{-1}\right)^{t_0},\\
&W=\prod_{s=1,\ldots ,m  }^{\longleftarrow}
R_{s0}( y_s e^{ hc }/x)^{-1} 
=R_{m0}( y_m e^{ hc }/x)^{-1}  \ldots
 R_{10}( y_1 e^{ hc }/x)^{-1}  
.
\end{align*}
Denote by $Z^\circ$ (resp.  $W^\circ$) the  product obtained from $Z$ (resp. $W$) by swapping its adjacent factors 
$\left(R_{k0}(xy_k e^{ hc })^{-1}\right)^{t_0}$ and $\left(R_{k+1\ts 0}(xy_{k+1} e^{ hc })^{-1}\right)^{t_0}$
(resp. $R_{k+1\ts 0}( y_{k+1} e^{ hc }/x)^{-1}$ and $R_{k0}( y_k e^{ hc }/x)^{-1}$).
 By the Yang--Baxter equation \eqref{trig_ybe_12}, we have
$$
U\ts Z =Z^\circ \ts U\Fand
U\ts W =W^\circ \ts  U.
$$
Therefore, we conclude that \eqref{tmpf12} coincides with
\beq\label{tmpf13}
	R_{[m]}^1 (y)^\circ \cdotrl 
\left(X^\circ \cdotrl\left(
 Y^\circ   \left(P_{k\ts k+1}\ts
B_{(m)}^1(y^\prime)\ts
P_{k\ts k+1}\right)    Z^\circ  \ts W^\circ \ts U\right) 
\right).
\eeq

The $R$-matrix product  $R_{[m]}^1 (y)^\circ $, given by \eqref{tmpf14}, can be written as
\begin{align*}
&\left(\prod_{i=k+2,\ldots ,m }^{\longleftarrow}
 R_{m i}^{t_i}  \ldots 
 R_{i+1\ts i}^{t_i}\right)
 \left(R_{m \ts k+1}^{t_{k+1}}  \ldots 
 R_{k+2\ts k+1}^{t_{k+1}} \right)
\left(R_{m k}^{t_k}  \ldots 
 R_{k+2\ts k}^{t_k} \right)
\\
&\times
\left(\prod_{i=1,\ldots ,k-1 }^{\longleftarrow}
 R_{m i}^{t_i}  \ldots R_{k+2 i}^{t_i} \ts 
R_{k  i}^{t_i} \ts R_{k+1 i}^{t_i} \ts
R_{k-1 i}^{t_i} \cdots
 R_{i+1\ts i}^{t_i}\right),
\end{align*}
By the Yang--Baxter equation \eqref{trig_ybe_12}, we have
$$
U\left(R_{m \ts k+1}^{t_{k+1}}  \ldots 
 R_{k+2\ts k+1}^{t_{k+1}} \right)
\left(R_{m k}^{t_k}  \ldots 
 R_{k+2\ts k}^{t_k} \right) =
\left(R_{m k}^{t_k}  \ldots 
 R_{k+2\ts k}^{t_k} \right)
\left(R_{m \ts k+1}^{t_{k+1}}  \ldots 
 R_{k+2\ts k+1}^{t_{k+1}} \right)
 U.
$$
This implies that \eqref{tmpf13} equals
\beq\label{tmpf16}
	\left(R_{[m]}^1 (y)^\bullet \cdotrl 
\left(X^\circ \cdotrl\left(
 Y^\circ   \left(P_{k\ts k+1}\ts
B_{(m)}^1(y^\prime)\ts
P_{k\ts k+1}\right)    Z^\circ  \ts W^\circ  \right) 
\right)\right) U,
\eeq
where
\begin{align*}
R_{[m]}^1 (y)^\bullet=\,
&\left(\prod_{i=k+2,\ldots ,m }^{\longleftarrow}
 R_{m i}^{t_i}  \ldots 
 R_{i+1\ts i}^{t_i}\right)
\left(R_{m k}^{t_k}  \ldots 
 R_{k+2\ts k}^{t_k} \right)
\left(R_{m \ts k+1}^{t_{k+1}}  \ldots 
 R_{k+2\ts k+1}^{t_{k+1}} \right)
\\
&\times
\left(\prod_{i=1,\ldots ,k-1 }^{\longleftarrow}
 R_{m i}^{t_i}  \ldots R_{k+2 i}^{t_i} \ts 
R_{k  i}^{t_i} \ts R_{k+1 i}^{t_i} \ts
R_{k-1 i}^{t_i} \cdots
 R_{i+1\ts i}^{t_i}\right).
\end{align*}

Finally, let
\beq\label{tmpf29}
T=
X^\circ \cdotrl\left(
 Y^\circ   \left(P_{k\ts k+1}\ts
B_{(m)}^1(y^\prime)\ts
P_{k\ts k+1}\right)    Z^\circ  \ts W^\circ\right),
\eeq
so that \eqref{tmpf16} can be written more briefly as 
\beq\label{tmpf17}
	\left(R_{[m]}^1 (y)^\bullet \cdotrl 
T\right) U.
\eeq
Denote by $R_{[m]}^1 (y)^\oplus $ the expression obtained from $R_{[m]}^1 (y)^\bullet$ by inserting the additional $R$-matrix 
$R_{k \ts k+1}^{t_{k+1}}
=R_{k \ts k+1}(y_k y_{k+1})^{t_{k+1}}
=R_{k+1 \ts k }(y_{k+1} y_{k })^{t_{k }}
$
right after 
$R_{k+2\ts k+1}^{t_{k+1}}$, i.e, 
\begin{align}
R_{[m]}^1 (y)^\oplus=\,
&\left(\prod_{i=k+2,\ldots ,m }^{\longleftarrow}
 R_{m i}^{t_i}  \ldots 
 R_{i+1\ts i}^{t_i}\right)
\left(R_{m k}^{t_k}  \ldots 
 R_{k+2\ts k}^{t_k} \right)
\left(R_{m \ts k+1}^{t_{k+1}}  \ldots 
 R_{k+2\ts k+1}^{t_{k+1}} \ts R_{k \ts k+1}^{t_{k+1}}\right)\non
\\
&\times
\left(\prod_{i=1,\ldots ,k-1 }^{\longleftarrow}
 R_{m i}^{t_i}  \ldots R_{k+2 i}^{t_i} \ts 
R_{k  i}^{t_i} \ts R_{k+1 i}^{t_i} \ts
R_{k-1 i}^{t_i} \cdots
 R_{i+1\ts i}^{t_i}\right).\label{tmpf30}
\end{align}
Clearly, this additional $R$-matrix satisfies 
$
V\cdotlr R_{k \ts k+1}^{t_{k+1}}= 1,
$
where $V$ is given by \eqref{tmpf20}.
Hence, \eqref{tmpf17} can be also written as
\beq\label{tmpf32}
\left(V\cdotlr \left(R_{[m]}^1 (y)^\oplus \cdotrl 
T\right)\right) U .
\eeq
Indeed, the symbol ``$\cdotrl$'' for the action of $R_{[m]}^1 (y)^\oplus$ indicates that, in particular, the $(k+1)$-th (resp. $k$-th) component of $R_{k \ts k+1}^{t_{k+1}}$ is applied from the right (resp. from the left) and the symbol ``$\cdotlr$''
for the action of $V$ that its $k$-th (resp. $(k+1)$-th) component  is applied from the left (resp. from the right), so that $V$ and $R_{k \ts k+1}^{t_{k+1}}$   cancel.
By examining the explicit expressions \eqref{tmpf29} and \eqref{tmpf30} for $T$ and $R_{[m]}^1 (y)^\oplus$, we conclude that \eqref{tmpf32} coincides with the image of the second term in \eqref{tmpf005} under \eqref{tmpf_001}, as required.
\end{prf}

Suppose $W$ is a restricted $\dyho$-module.
 Then the expression
$$
\bc(z)_W  = B^+(z)_W  B^*(ze^{hc} )_W^{-1}
$$
is a well-defined  element of $\ndo\CC^N\ot\om(W,W((z))_h)$.
The operator series $\bc(z)_W $ satisfies the identity which closely resembles the {\em quantum current commutation relation} of Reshetikhin and Semenov-Tian-Shansky \cite{RS},
\beq\label{qccr}
\bc_1(x)_W\ts R_{21}(ye^{-2hc}/x)\ts \bc_2(y)_W\ts R_{21}(y/x)^{-1}
\equiv
R(x/y)^{-1}\bc_2(y)_W\ts R(xe^{-2hc}/y)\ts \bc_1(x)_W.
\eeq
The   symbol ``$\equiv$'' indicates that for any positive integer $n$ there exists  $p(x,y)\in\CC[[x,y]]$ satisfying $p(ye^x,y)\neq 0$ such that
both sides of \eqref{qccr},  when multiplied by $p(x,y)$,  coincide modulo $h^n$.
Our next goal is to generalize this property of $\bc(z)_W$. For any $n\geqslant 1$ and the family of variables $x=(x_1,\ldots ,x_n),$ define
\beq\label{beceenovi}
\bc_{[n]} (x)_W = B^+_{[n]}(x)_W  B^*_{[n]}(xe^{ hc} )_W^{-1}, \quad\text{where}\quad xe^{ hc}=(x_1 e^{ hc},\ldots ,x_n e^{ hc}).
\eeq
As before, we assume that $W$ is a restricted $\dyho$-module. This ensures that $\bc_{[n]} (x)_W$ is a well-defined   element of
$(\ndo\CC^N)^{\ot n}\ot\om(W,W((x_1,\ldots ,x_n))_h),$ where $W((x_1,\ldots ,x_n))_h$ is the $h$-adic completion of $W((x_1,\ldots ,x_n)) $.
Next, for any $m\geqslant 1,$  the family of variables $y=(y_1,\ldots ,y_m)$  and $a\in\CC$ define
\beq\label{rnm1opp}
R_{nm}^{21}(ye^{ah}/x)=
\prod_{r=1,\dots,n }^{\longleftarrow} \prod_{s=n+1,\dots,n+m }^{\longrightarrow}
R_{sr}( y_{s-n} e^{ah}/x_r).
\eeq
Note that, in comparison with \eqref{rnm1},    the factors in \eqref{rnm1opp}   are given  in the opposite order.

\begin{pro}\label{prop_aux}
Let $W$ be a restricted $\dyho$-module. For any integers $m,n\geqslant 1$ and    $x=(x_1,\ldots x_n)$ and $y=(y_1,\ldots y_m)$, we have
\begin{align}
&\bc_{[n+m]}(x_1,\ldots x_n,y_1,\ldots y_m)_W\label{qaargen}\\
\equiv\, &   \bc_{[n]}^{13}(x)_W\ts R_{nm}^{21}(ye^{-2hc}/x)\ts \bc_{[m]}^{23}(y)_W\ts R_{nm}^{21}(y/x)^{-1}\label{qbbrgen}\\
\equiv\,&  
R_{nm}^{12}(x/y)^{-1}\bc_{[m]}^{23}(y)_W\ts R_{nm}^{12}(xe^{-2hc}/y)\ts \bc_{[n]}^{13}(x)_W,\non
\end{align}
where the indices $ 1,2,3$ denote the tensor factors as follows:
$$
\smalloverbrace{(\ndo\CC^N)^{\ot n}}^{1}
\ot
\smalloverbrace{(\ndo\CC^N)^{\ot m}}^{2}
\ot
\smalloverbrace{\ndo W}^{3}.
$$
The   symbol ``$\equiv$'' indicates that for any positive integer $k$ there exist  $p_{i,j}(x_i,y_j)\in\CC[[x_i, y_j]]$, $i=1,\ldots ,n$, $j=1,\ldots ,m$, satisfying
$p_{i,j}(y_j e^{x_i},y_j)\neq 0$,
 such that
the corresponding expression,  when multiplied by 
$$
p(x_1,\ldots ,x_n,y_1,\ldots ,y_m)\coloneqq\prod_{i=1,\ldots,n}\prod_{j=1,\ldots,m}p_{i,j}(x_i,y_j)\in\CC[[x_1,\ldots ,x_n,y_1,\ldots , y_m]],
$$
  coincide modulo $h^k$. 
\end{pro}

\begin{rem}\label{rem_AB}
Later on, we shall use the following simple observation. The   relation between \eqref{qaargen} and  \eqref{qbbrgen} in Proposition \ref{prop_aux} can be  also written as
	$$
	X\cdotrl\left(\bc_{[n+m]}(x_1,\ldots x_n,y_1,\ldots y_m)_W\ts Z^{-1}\right)
	\equiv  \bc_{[n]}^{13}(x)_W\ts   \bc_{[m]}^{23}(y)_W,
	$$
	where
	$Z=R_{nm}^{21}(y/x)^{-1}$
	and $X$ is chosen so that we have
	$$
	X\cdotrl R_{nm}^{21}(ye^{-2hc}/x)=1,
	$$
	where  ``$\cdotrl$'' denotes the standard multiplication in $((\ndo\CC^N)^{\text{op}})^{\ot n}\ot(\ndo\CC^N)^{\ot m}$.

\end{rem}

Our next goal is to establish a connection between restricted $\dyho$-modules and the Etingof--Kazhdan quantum affine vertex algebra $\Vc$ \cite[Thm. 2.3]{EK} associated with the  $R$-matrix   \eqref{erofu}. To carry this out, we shall use 
 the notion of $\phi$-coordinated quasi  module. For the reader convenience, we recall its definition, which is  
  an $h$-adic counterpart of \cite[Def. 3.4]{Liphi} for the associate $\phi=\phi(z_2,z_0)=z_2 e^{z_0}$ . 

\begin{defn}\label{def_ver12}
Let $(V,Y,\vac,\Sc)$ be a quantum vertex algebra.
A {\em $\phi$-coordinated quasi $V$-module} is a pair $(W,Y_W)$ such that $W$ is a topologically free $\CC[[h]]$-module and $ Y_W(\cdot, z) $ is  a $\mathbb{C}[[h]]$-module map
\begin{align*}
Y_W(\cdot, z) \colon V &\to \om(W,W((z))_h)\\
v &\mapsto Y_W(v,z)   =\sum_{r\in\mathbb{Z}} v_r   \ts z^{-r-1} 
\end{align*}
which satisfies  
 $ Y_W(\vac,z)w=w$ for all $w\in W$
and the {\em  quasi weak associativity}: for any $u,v\in V$ and $k>0$ there exists $p(z_1,z_2)\in\CC[[z_1,z_2]]$ satisfying $p(z_2 e^{z_0},z_2  )\neq 0$ such that
\begin{align}
&p(z_1,z_2) \ts Y_W(u,z_1)Y_W(v,z_2)\in\om (W,W((z_1,z_2)) )\mod h^k\Fand\label{qwassoc1}\\
&\big(p(z_1,z_2) \ts Y_W(u,z_1)Y_W(v,z_2)\big)\big|_{z_1= z_2e^{z_0}}^{\modd h^k}  \big. \label{qwassoc2qwassoc2} \\
&\qquad- p(z_2 e^{z_0},z_2  )\ts Y_W\left(Y(u,z_0)v,z_2\right)\ts
\in\ts  h^k \om(W,W[[z_0^{\pm 1},z_2^{\pm 1}]]).  \label{qwassoc2}
\end{align}
\end{defn}

 Regarding Definition \ref{def_ver12}, note that the constraint imposed by
\eqref{qwassoc1} requires that the given expression belongs to $\om (W,W((z_1,z_2)) )$ when regarded modulo $h^k$. Also, the symbol ``mod $h^k$'' in the superscript in \eqref{qwassoc2qwassoc2} means that the corresponding substitution $z_1=z_2e^{z_0}$ should be applied to the given expression regarded modulo $h^k$ (while \eqref{qwassoc1} ensures that such a substitution exists).
 
\begin{thm}\label{thm_33}
Let $W$ be a restricted $\dyho$-module.
There exists a unique structure of $\phi$-coordinated quasi $\Vc$-module on $W $ so that the module map $Y_{W}(\cdot , z)$ satisfies
\beq\label{modulemap}
Y_{W}(T_{[n]}^+(u_1,\ldots ,u_n)\vac , z)=
\bc_{[n]}(x_1, \ldots ,x_n )_W\big|_{x_1=ze^{u_1},\ldots ,x_n=ze^{u_n}}.
\eeq
\end{thm}

\begin{prf}
Let $W$ be a restricted $\dyho$-module. First, we show that the map $Y_W(\cdot, z)$ is well-defined by \eqref{modulemap}. Note that the defining relations \eqref{rtt} for the algebra $\textrm{U}(R)$ can be written equivalently as
\beq\label{rbartt}
\R(e^{u-v})\ts T^+_{1} (u)\ts T^+_2  (v)=  T^+_2  (v)\ts T^+_{1} (u)\ts \R(e^{u-v}),
\eeq
where the $R$-matrix $\R(e^u)\in\ndo\CC^N\ot\ndo\CC^N[[u]]$ is obtained by setting $z=e^u$ in
\eqref{rbarmatrix}. 
Let $1\leqslant k<n$ be integers,  
$$u=(u_1,\ldots ,u_n)\fand u^\prime=(u_1^\prime,\ldots,u_n^\prime)=(u_1,\ldots u_{k-1},u_{k+1},u_k,u_{k+2},\ldots,u_n)$$
 the families of variables and $P_{k\ts k+1}=1^{\ot(k-1)}\ot P\ot 1^{\ot(n-k-1)}$   the action of the permutation operator
\eqref{permutationoperator}
on the tensor factors $k$ and $k+1$ of $(\ndo\CC^N)^{\ot n}$.
It suffices to verify that the ideal of defining relation \eqref{rbartt} belongs to the kernel of $Y_W(\cdot, z)$, or, in other words, that the images of
$$
\R_{k\ts k+1}(e^{u_k - u_{k+1}})\ts T_{[n]}^+(u )\vac
\Fand
P_{k\ts k+1}\ts T_{[n]}^+(u^\prime )\vac\ts P_{k\ts k+1}\ts \R_{k\ts k+1}(e^{u_k - u_{k+1}})
$$
under \eqref{modulemap} coincide. Clearly, their images are given by
$$
\R_{k\ts k+1}(e^{u_k - u_{k+1}})\ts
\bc_{[n]}(x_1, \ldots ,x_n )_W\big|_{x_1=ze^{u_1},\ldots ,x_n=ze^{u_n}}
$$
and
$$
P_{k\ts k+1}\ts
\bc_{[n]}(x_1, \ldots ,x_n )_W\big|_{x_1=ze^{u^\prime_1},\ldots ,x_n=ze^{u^\prime_n}}
\ts P_{k\ts k+1}\ts \R_{k\ts k+1}(e^{u_k - u_{k+1}}) 
$$
respectively. The fact that they coincide can be proved by an argument which follows the proof of \cite[Lemma 3.6]{KS} and relies on the   Yang--Baxter equation and the reflection equations \eqref{re1} and \eqref{re2}.

We already observed   that  \eqref{beceenovi} belongs to
$(\ndo\CC^N)^{\ot n}\ot\om(W,W((x_1,\ldots ,x_n))_h) $
if $W$ is a restricted $\dyho$-module. This implies that the image of the map
$Y_W(\cdot, z)$ belongs to $\om(W,W((z))_h)$. Therefore, to prove the theorem, it remains to check that  the map
$Y_W(\cdot, z)$ possesses the quasi weak associativity properties 
\eqref{qwassoc1}--\eqref{qwassoc2}.

For any $n $ and $m$ and the variables $u=(u_1,\ldots, u_n)$, $v=(v_1,\ldots ,v_m)$ consider the power series
\beq\label{tmnt}
T_{[n]}^{+13}(u_1,\ldots, u_n)\ts T_{[m]}^{+24}(v_1,\ldots ,v_m)(\vac\ot\vac),
\eeq
where the superscripts $ 1,2,3,4$ denote the tensor factors as follows:
\beq\label{tmnt2}
\smalloverbrace{(\ndo\CC^N)^{\ot n}}^{1}
\ot
\smalloverbrace{(\ndo\CC^N)^{\ot m}}^{2}
\ot
\smalloverbrace{\Vc}^{3}
\ot
\smalloverbrace{\Vc}^{4} 
\eeq
and     $\vac\ot\vac$ belongs to the tensor factors $n+m+1$ and $n+m+2$ of \eqref{tmnt2}. Clearly, it suffices to show that \eqref{qwassoc1}--\eqref{qwassoc2} hold when the corresponding expressions are applied on the coefficients of  \eqref{tmnt}.  By \eqref{modulemap}, the application of $Y_W(\cdot ,z_1)Y_W(\cdot, z_2)$ on \eqref{tmnt} yields
\begin{align}
&\bc_{[n]}^{13}(x_1, \ldots ,x_n )_W\big|_{x_1=z_1 e^{u_1},\ldots ,x_n=z_1 e^{u_n}} \ts
\bc_{[m]}^{23}(y_1, \ldots ,y_m )_W\big|_{y_1=z_2 e^{v_1},\ldots ,y_m=z_2 e^{v_m}}\non\\
=\,&\bc_{[n]}^{13}(x_1, \ldots ,x_n )_W  \ts
\bc_{[m]}^{23}(y_1, \ldots ,y_m )_W\big|_{x_1=z_1 e^{u_1},\ldots ,x_n=z_1 e^{u_n}, y_1=z_2 e^{v_1},\ldots ,y_m=z_2 e^{v_m}}.\label{tmnt3}
\end{align}
Since
$$
\bc_{[n+m]}(x_1,\ldots x_n,y_1,\ldots y_m)_W\in(\ndo\CC^N)^{\ot (n+m)}
\ot\om(W,W((x_1,\ldots x_n,y_1,\ldots y_m))_h),
$$
the first assertion \eqref{qwassoc1} of the quasi weak associativity
follows from the relation between \eqref{qaargen} and \eqref{qbbrgen},  as given by Proposition \ref{prop_aux}.

Let us apply $Y_W(Y(\cdot, z_0), z_2)$ to \eqref{tmnt}. By using \eqref{EK_vo} and \eqref{modulemap}, we obtain
\beq\label{tmnt4}
X\cdotrl \left(\bc_{[n+m]}(x_1,\ldots,x_n,y_1,\ldots ,y_m)_W|_{x_1=z_2 e^{z_0+u_1},\ldots ,x_n=z_2 e^{z_0+u_n}, y_1=z_2 e^{v_1},\ldots ,y_m=z_2 e^{v_m}}Z^{-1}\right),
\eeq
where
$$
Z=R_{nm}^{12}(e^{z_0 +u-v})=
\prod_{r=1,\dots,n }^{\longrightarrow} \prod_{s=n+1,\dots,n+m }^{\longleftarrow}
R_{rs}(e^{z_0+u_r - v_{s-n}})
$$
and $X$ is chosen so that we have
$$
X\cdotrl R_{nm}^{12}(e^{z_0 +u-v+2hc})^{-1} =1,
$$
	where  ``$\cdotrl$'' denotes the standard multiplication in $((\ndo\CC^N)^{\text{op}})^{\ot n}\ot(\ndo\CC^N)^{\ot m}$.
The second assertion  of the quasi weak associativity can be now deduced by comparing the expressions in \eqref{tmnt3} and \eqref{tmnt4}, which correspond to the   terms  \eqref{qwassoc2qwassoc2} and \eqref{qwassoc2} respectively,  and by using   Remark \ref{rem_AB} and the unitarity property \eqref{unitarity}.
\end{prf}

\section{Invariants of the extended twisted \texorpdfstring{$h$}{h}-Yangian}\label{section_04}

By Theorem \ref{thm_25}, $\yhg$ is equipped with the structure of restricted $\dyho$-module. Hence, by applying Theorem \ref{thm_33} to  $W=\yhg$, we obtain a structure of $\phi$-coordinated quasi $\Vc$-module on $ \yhg$. We shall denote this $\phi$-coordinated quasi  module by $\Mcc$ and the corresponding module map by $Y_{\Mcc}(\cdot, z)$. In addition, for $c=-N/2$ we 
denote $\Mcc$ (resp. $Y_{\Mcc }(\cdot, z)$)
by $\Mccrit$ (resp. $Y_{\Mccrit}(\cdot, z)$), and the algebra $\dyho$    by $\dyhocrit$. 
Define the {\em submodule of invariants} of $\Mcc$ by
$$
\mathfrak{z}(\Mcc)=\left\{a\in\Mcc\,:\, B^*(z)_{\Mcc} a=a\right\}.
$$
Note that we have
$$
Y_{\Mcc}(v,z)a\in\Mcc[[z]]\quad\text{for all }v\in\Vc\text{ and } a\in\mathfrak{z}(\Mcc).
$$

The goal of this section is to investigate the submodule of invariants  $\mathfrak{z}(\Mccrit)$, its relation with the center of $\dyhocrit$ and its applications. Our main tool will be the fusion procedure   \cite{C,IMO,N} for the $R$-matrix
$$
\Rv(z)=\frac{e^{h/2}-e^{-h/2}z}{\left(1-z\right)\left(1-ze^{-h}\right)}   \R(z) ,
$$
where $\R(z)$ is defined by \eqref{rbarmatrix}. 
First, we recall the fusion procedure by following  the exposition in \cite[Sections 2, 3]{JLM0}.
Suppose  $\Lambda$ is a standard tableau of shape $\lambda\vdash n$. Denote by  $c_k(\Lambda)$ the content $j-i$ of the box $(i,j)$ of $\lambda$ occupied by $k$ in $\Lambda.$
Introduce the order on the set of all pairs $(i,j)$, where 
$1\leqslant i <j\leqslant n,$ as follows:
\beq\label{order}
(i,j )\prec (i^\prime, j^\prime)
\qquad\text{if}\qquad
j<j^\prime
\quad\text{or}\quad
j=j^\prime\text{ and }i<i^\prime .
\eeq
Consider the expression
$$
\Rv_\Lambda(z_1,\ldots ,z_n)
=
\prod_{(i,j) }^{\longrightarrow} 
\left(
P_{j-i\ts j-i+1}\ts \Rv_{j-i\ts j-i+1}(z_i e^{ (c_i(\Lambda)-c_j(\Lambda))h}/z_j)
\right),
$$
where the arrow over the product symbol indicates that the product over the set of all pairs $(i,j)$ with $1\leqslant i <j\leqslant n$ is ordered with respect to \eqref{order} and $P_{rs}$ is the action of the permutation operator \eqref{permutationoperator}
on the $r$-th and $s$-th tensor factor of $(\ndo\CC^N)^{\ot n}.$
Denote by $\lambda^\prime$   the conjugate partition of $\lambda,$ $c_{\lambda^\prime}$ the Schur element corresponding to $\lambda^\prime$ and $\check{R}_0$ a certain invertible operator on $(\CC^N)^{\ot n}$; see\cite{JLM0} for more details. The fusion procedure implies that  
$$
\Ec_\Lambda = \frac{1}{c_{\lambda^\prime}}\ts \Rv_\Lambda(z_1,\ldots ,z_n)\ts
\check{R}_0^{-1}\Big|_{z_1=1}\Big|_{z_2=1}\ldots \Big|_{z_n=1},
$$ 
is a well-defined operator satisfying $\Ec_\Lambda^2=\Ec_\Lambda.$

Denote by $\Vccrit$ the quantum   vertex algebra $\Vcdva$ at the critical level $c=-N$.
Consider the   power series
\beq\label{telambda}
T_\Lambda(u)=\tr_{1,\ldots ,n} \, T_{[n]}^+ (u -c_1(\Lambda)h,\ldots ,u -c_n(\Lambda)h)\vac \ts D^{\ot n}\ts \Ec_\Lambda \in \Vccrit[[u]],
\eeq
where the matrix $D$ is given by \eqref {diagonal}.
By \cite[Prop. 7.4]{BJK}, its coefficients belong to the center of $\Vccrit$. Let $T_\Lambda=T_\Lambda(0)$ be the constant term  of the series \eqref{telambda}.

Recall that, due to  Theorem \ref{thm_33}, $\Mccrit$ is a $\phi$-coordinated quasi $\Vccrit$-module. By applying the corresponding module map  $Y_{\Mccrit}(\cdot, z)$,  given by \eqref{modulemap} with $W=\Mccrit$,  to $T_\Lambda$, one obtains
\begin{align} 
 Y_{\Mccrit}(T_\Lambda , z) =  &\, \tr_{1,\ldots ,n} \, \bc_{[n]}(x_1, \ldots ,x_n )_{\Mccrit}\big|_{x_1=ze^{-c_1(\Lambda)h},\ldots ,x_n=ze^{-c_n(\Lambda)h}} \ts D^{\ot n}\ts \Ec_\Lambda\non\\
 = & \, \tr_{1,\ldots ,n} \, B^+_{[n]}(ze^{-c(\Lambda)h})_{\Mccrit} \ts B^*_{[n]}(ze^{-c(\Lambda)h -Nh/2})_{\Mccrit}^{-1}\ts D^{\ot n}\ts \Ec_\Lambda,\label{central}
\end{align}
where
\begin{align}
&ze^{-c(\Lambda)h}=
(ze^{-c_1(\Lambda)h}, \ldots ,ze^{-c_n(\Lambda)h}),\label{zeovi} \\
&ze^{-c(\Lambda)h-Nh/2}=
(ze^{-c_1(\Lambda)h-Nh/2}, \ldots ,ze^{-c_n(\Lambda)h-Nh/2}).\non
\end{align}
Next, applying $Y_{\Mccrit}(T_\Lambda , z)$ to $1\in \Mccrit$ yields
 \begin{align}\label{belambda_10}
 B_\Lambda(z)\coloneqq Y_{\Mccrit}(T_\Lambda , z)1  
  \in\Mccrit[[z^{\pm 1}]].
\end{align}
Finally, motivated by the form of \eqref{central}, we introduce the formal power series
$$
\mathcal{B}_\Lambda(z)=
\tr_{1,\ldots ,n} \, B^+_{[n]}(ze^{-c(\Lambda)h})  \ts B^*_{[n]}(ze^{-c(\Lambda)h -Nh/2})^{-1}\ts D^{\ot n}\ts \Ec_\Lambda\in \dyhocritt[[z^{\pm 1}]].
$$ 
Clearly, we have
\beq\label{blambdablambda}   
B_\Lambda(z) =\mathcal{B}_\Lambda(z)1 ,
\eeq
 where $1\in \Mccrit\equiv\yhg$ and the action of $\mathcal{B}_\Lambda(z)$  is given by Theorem \ref{thm_25}.

\begin{lem} 
Let   $x=ze^{-c(\Lambda)h}$ be the family of variables given by \eqref{zeovi}. We have
\beq\label{bneccc}
\mathcal{B}_{[n]}(x)\ts\Ec_\Lambda =\Ec_\Lambda \ts \mathcal{B}_{[n]}(x)\ts\Ec_\Lambda .
\eeq
\end{lem}

\begin{prf}  By the proof of Theorem \ref{thm_33}, the assignments
\beq\label{jjustamap}
 T_{[m]}^+(u_1,\ldots ,u_m)\vac \mapsto
\bc_{[m]}(v_1, \ldots ,v_m ) \big|_{v_1=ze^{u_1},\ldots ,v_m=ze^{u_m}}
\eeq
define a $\CC[[h]]$-module map $\Vccrit \to \dyhocritt[[z^{\pm 1}]]$. Moreover, by \cite[Lemma 7.1 (a)]{BJK}, which is based on \cite[Lemma 3.2]{JLM0}, we have
$$
T_{[n]}^+(u-c_1(\Lambda)h,\ldots, u-c_n(\Lambda)h)\vac\ts \Ec_\Lambda = \Ec_\Lambda \ts T_{[n]}^+(u-c_1(\Lambda)h,\ldots, u-c_n(\Lambda)h)\vac \ts \Ec_\Lambda .
$$
Thus, in particular, by setting $u=0$ we get
$$
T_{[n]}^+( -c_1(\Lambda)h,\ldots,  -c_n(\Lambda)h)\vac\ts \Ec_\Lambda = \Ec_\Lambda \ts T_{[n]}^+( -c_1(\Lambda)h,\ldots,  -c_n(\Lambda)h)\vac \ts \Ec_\Lambda .
$$
By applying the map \eqref{jjustamap} to this identity, we obtain \eqref{bneccc}, as required.
\end{prf}

In the next lemma,  the superscripts $1,2,3$   denote the following tensor factors:
\beq\label{superscripts_meaning}
\smalloverbrace{(\ndo\CC^N)^{\ot n}}^{1}
\ot
\smalloverbrace{(\ndo\CC^N) }^{2}
\ot
\smalloverbrace{\dyhocritt}^{3} .
\eeq
The terms  $D^{\ot n}  $ and $\Ec_\Lambda$ are applied to the   factor  of \eqref{superscripts_meaning} denoted by the superscript $1$.
  By $t_i$ with $i=1,\ldots ,n$ we denote the action of the transposition $t$ on the $i$-th tensor factor  of \eqref{superscripts_meaning}, so that all transposition are applied to the   factor  of \eqref{superscripts_meaning} denoted by the superscript $1$ as well. 
\begin{lem}\label{lemma_41}
Let   $x= ze^{-c(\Lambda)h}$ be the family of variables given by \eqref{zeovi} and $y$ a single variable. Define
\begin{align*}
&X_+ =\left(R_{n1}^{21}(ye^{Nh}/x)^\prime\right)^{-1},
&& Y_+=\RRR_{n1}^{21}(xye^{-Nh})^{\prime},\\
&Z_+=\left(R_{n1}^{21}(y /x)^\prime\right)^{-1}, && W_+=\RRR_{n1}^{21}(xy )^{\prime}.
\end{align*}
\begin{enumerate}
\item
We have
\begin{gather}
 \bc_{[n]}^{13}(x)\ts X_+\ts B^{+23}(y)\ts Y_+ = 
Z_+\ts B^{+23}(y)\ts  W_+\ts \bc_{[n]}^{13}(x).\label{identity1}
\end{gather}
\item  
Let $(X_+)^\sim$ be the inverse of $X_+$ with respect to the standard multiplication ``$\cdotrl$'' in $((\ndo\CC^N)^{\ot n})^{\text{op}}\ot  \ndo\CC^N $, which is uniquely determined by the property $(X_+)^\sim \cdotrl X_+=1$.
We have
\begin{align}
&
(X_+)^\sim
D^{\ot n}\ts
Z_+ =D^{\ot n}, \label{identitya}\\
&  
(W_+)^{t_1,\ldots ,t_n}
 D^{\ot n}
\left((Y_+)^{-1} \right)^{t_1,\ldots ,t_n}=D^{\ot n}, \label{identityb}
\end{align}

\item
For any $U\in\left\{(X_+)^\sim, (Y_+)^{-1}, Z_+, W_+ \right\}$ we have
\beq\label{eueue}
U\ts\Ec_\Lambda =\Ec_\Lambda \ts U\ts\Ec_\Lambda .
\eeq
\end{enumerate}
\end{lem}

\begin{prf}
(1) The identity can be verified by a straightforward computation which relies on the relations \eqref{re_gen1} and  \eqref{re_gen3} for $m=1$. 

\noindent  (2) Both equalities follow from the crossing symmetry properties \eqref{csym}.

\noindent (3) This follows from \cite[Lemma 7.1 (b), (c)]{BJK}, which is based on \cite[Lemma 3.3]{JLM0}.
\end{prf}

In the next lemma,  the superscripts $0,1,2 $   denote the following tensor factors:
\beq\label{superscripts_meaningb}
\smalloverbrace{(\ndo\CC^N) }^{0}
\ot
\smalloverbrace{(\ndo\CC^N)^{\ot n}}^{1}
\ot
\smalloverbrace{\dyhocritt}^{2} .
\eeq
The terms  $D^{\ot n}  $ and $\Ec_\Lambda$ are applied to the   factor  of \eqref{superscripts_meaningb} denoted by the superscript $1$. 
 By $t_i$ with $i=1,\ldots ,n$ we denote the action of the transposition $t$ on the $(i+1)$-th tensor factor  of \eqref{superscripts_meaningb}, so that all transposition are applied to the   factor  of \eqref{superscripts_meaningb} denoted by the superscript $1$ as well. 
We omit the proof of the lemma as it  proceeds analogously to the proof of Lemma \ref{lemma_41}.

\begin{lem}\label{lemma_42}
Let   $x= ze^{-c(\Lambda)h}$ be the family of variables defined by \eqref{zeovi} and $y$ a single variable. Define
\begin{align*}
&X_*=R_{1n}^{10}(xe^{-Nh/2}/y)^\prime,&&Y_*=\RRR_{1n}^{10}( xy e^{ - Nh/2} )^\prime , \\
&Z_*=R_{1n}^{10}(xe^{ Nh/2}/y)^\prime,&&W_*=\RRR_{1n}^{10}( xy e^{  Nh/2}  )^\prime.
\end{align*}
\begin{enumerate}
\item
We have
\begin{gather*}
Z_*\ts B^{*02}(y)\ts  W_*\ts \bc_{[n]}^{12}(x)=
\bc_{[n]}^{12}(x)\ts X_*\ts B^{*02}(y)\ts Y_*
. 
\end{gather*}
\item 
Let $(W_*)^\sim$ be the inverse of $W_*$ with respect to the standard multiplication ``$\cdotrl$'' in $(\ndo\CC^N)^{\text{op}}   \ot  (\ndo\CC^N)^{\ot n}  $, which is uniquely determined by the property $(W_*)^\sim \cdotrl W_*=1$.
We have
\begin{align*}
&\left((Z_*)^{-1}\right)^{t_1,\ldots ,t_n}
D^{\ot n}\ts
(X_*)^{t_1,\ldots ,t_n} =D^{\ot n},  \\
&  
Y_*\ts
 D^{\ot n}\ts
(W_*)^\sim=D^{\ot n}.  
\end{align*} 
\item
For any $U\in\left\{  X_* ,  Y_* , (Z_*)^{-1}, (W_*)^\sim\right\}$ we have
$$
U\ts\Ec_\Lambda =\Ec_\Lambda \ts U\ts\Ec_\Lambda .
$$
\end{enumerate}
\end{lem}

The following theorem is the main result of this section.

\begin{thm}\label{thm_41}
\begin{enumerate}
\item  All coefficients of   $\mathcal{B}_\Lambda(z)$ belong to the center of   $\dyhocritt$.
\item All coefficients of   $B_\Lambda(z)$ belong to the submodule of invariants $\mathfrak{z}(\Mccrit)$.
\end{enumerate}
\end{thm}

\begin{prf}
(1) 
Clearly, it suffices to check that $\mathcal{B}_\Lambda(z)$ commutes with $B^+(y)$ and $B^*(y)$.
Let us prove that $\mathcal{B}_\Lambda(z)$ commutes with $B^+(y)$. 
Denote by the superscripts $1,2,3$   the tensor factors as in \eqref{superscripts_meaning}. Throughout the proof, we use the notation from Lemma \ref{lemma_41}.
We have
\begin{align}
\mathcal{B}_\Lambda(z)\ts B^+(y) = &\,
\tr_{1,\ldots ,n} \, \mathcal{B}^{ 13}_{[n]}(ze^{-c(\Lambda)h}) \ts D^{\ot n}\ts \Ec_\Lambda\ts B^{+23}(y)\non\\
=&\,\tr_{1,\ldots ,n} \, \mathcal{B}^{ 13}_{[n]}(ze^{-c(\Lambda)h}) \ts B^{+23}(y)\ts D^{\ot n}\ts \Ec_\Lambda,\label{racun1}
\end{align}
where $ze^{-c(\Lambda)h}$ is given by \eqref{zeovi}. 
By using \eqref{identity1}, we rewrite \eqref{racun1} as
$$
 \tr_{1,\ldots ,n} \,
\left((X_+)^\sim\cdotrl\left(
Z_+\ts
B^{+23}(y)\ts 
W_+\ts 
\bc_{[n]}^{13}(ze^{-c(\Lambda)h})\ts (Y_+)^{-1}\right)\right)  D^{\ot n}\ts \Ec_\Lambda.
$$
Next, by the cyclic property of the trace, this equals
$$
 \tr_{1,\ldots ,n} \,
 (X_+)^\sim \ts D^{\ot n}\ts \Ec_\Lambda  \ts
Z_+\ts
B^{+23}(y)\ts 
W_+\ts 
\bc_{[n]}^{13}(ze^{-c(\Lambda)h})\ts (Y_+)^{-1}  .
$$
As the idempotent $ \Ec_\Lambda$ commutes with $D^{\ot n}$, we can use \eqref{eueue}, \eqref{bneccc} and the cyclic property of the trace to rewrite the above expression as
\beq\label{rn_123}
 \tr_{1,\ldots ,n} \,
  \Ec_\Lambda (X_+)^\sim \ts D^{\ot n}\ts \Ec_\Lambda  \ts
Z_+\ts \Ec_\Lambda \ts
B^{+23}(y)\ts \Ec_\Lambda\ts 
W_+\ts \Ec_\Lambda\ts 
\bc_{[n]}^{13}(ze^{-c(\Lambda)h})\ts \Ec_\Lambda\ts (Y_+)^{-1}  \ts \Ec_\Lambda.
\eeq
By  \eqref{eueue} we have
$ 
\Ec_\Lambda  \ts
Z_+\ts \Ec_\Lambda
=
Z_+\ts \Ec_\Lambda,
$ 
so that \eqref{rn_123} is equal to
\begin{align*}
&\tr_{1,\ldots ,n} \,
  \Ec_\Lambda (X_+)^\sim \ts D^{\ot n}\ts  
Z_+\ts \Ec_\Lambda \ts
B^{+23}(y)\ts \Ec_\Lambda\ts 
W_+\ts \Ec_\Lambda\ts 
\bc_{[n]}^{13}(ze^{-c(\Lambda)h})\ts \Ec_\Lambda\ts (Y_+)^{-1}  \ts \Ec_\Lambda\\
=&\,
\tr_{1,\ldots ,n} \,
  \Ec_\Lambda  \ts D^{\ot n} \ts \Ec_\Lambda\ts
B^{+23}(y)\ts \Ec_\Lambda\ts 
W_+\ts \Ec_\Lambda\ts 
\bc_{[n]}^{13}(ze^{-c(\Lambda)h})\ts \Ec_\Lambda\ts (Y_+)^{-1}  \ts \Ec_\Lambda,
\end{align*}
where the equality follows from \eqref{identitya}. By employing the cyclic property of the trace to move the rightmost copy of $ \Ec_\Lambda$ to the left and then using \eqref{eueue}, we bring the above expression to
\begin{align}
&\tr_{1,\ldots ,n} \,
    D^{\ot n} \ts  
B^{+23}(y)\ts  
W_+\ts  
\bc_{[n]}^{13}(ze^{-c(\Lambda)h})\ts \Ec_\Lambda\ts (Y_+)^{-1}   \non\\
=&\,
\tr_{1,\ldots ,n} \,
B^{+23}(y)\ts  
W_+\ts  
\bc_{[n]}^{13}(ze^{-c(\Lambda)h})\ts \Ec_\Lambda\ts (Y_+)^{-1} \ts D^{\ot n} \non \\ 
=&\,
B^{+ }(y)\ts 
\tr_{1,\ldots ,n} \,
W_+\ts  
\bc_{[n]}^{13}(ze^{-c(\Lambda)h})\ts \Ec_\Lambda\ts (Y_+)^{-1} \ts D^{\ot n}. \label{rn_124} 
\end{align}
Let
$$
G=W_+\Fand H= \bc_{[n]}^{13}(ze^{-c(\Lambda)h})\ts \Ec_\Lambda\ts (Y_+)^{-1} \ts D^{\ot n}.
$$
Since
$$
\tr_{1,\ldots ,n} \, G\ts H
=\tr_{1,\ldots ,n} \, G^{t_1,\ldots ,t_n}\ts H^{t_1,\ldots ,t_n},
$$
we conclude that \eqref{rn_124} equals
$$
B^{+ }(y)\ts 
\tr_{1,\ldots ,n} \,
(W_+)^{t_1,\ldots ,t_n}\ts  D^{\ot n}\ts ((Y_+)^{-1})^{t_1,\ldots ,t_n}
\left(\bc_{[n]}^{13}(ze^{-c(\Lambda)h})\ts \Ec_\Lambda\right)^{t_1,\ldots ,t_n}. 
$$
Finally, by \eqref{identityb}, this is equal to
\begin{align*}
&B^{+ }(y)\ts 
\tr_{1,\ldots ,n} \,
   D^{\ot n} 
\left(\bc_{[n]}^{13}(ze^{-c(\Lambda)h})\ts \Ec_\Lambda\right)^{t_1,\ldots ,t_n}\\
=&\, B^{+ }(y)\ts 
\tr_{1,\ldots ,n} \,
 \bc_{[n]}^{13}(ze^{-c(\Lambda)h})\ts \Ec_\Lambda \ts  D^{\ot n} \\
=&\,   B^+(y)\ts\mathcal{B}_\Lambda(z),
\end{align*}
as required.

It remains to prove that $\mathcal{B}_\Lambda(z)$ commutes with $B^*(y)$. 
Denote by the superscripts $0,1,2 $   the tensor factors as in \eqref{superscripts_meaningb}. 
We have
\begin{align*}
B^*(y)\ts \mathcal{B}_\Lambda(z)   = &\,
\tr_{1,\ldots ,n} \,B^{+02}(y)\ts \mathcal{B}^{ 12}_{[n]}(ze^{-c(\Lambda)h}) \ts D^{\ot n}\ts \Ec_\Lambda   , 
\end{align*}
where $ze^{-c(\Lambda)h}$ is given by \eqref{zeovi}. One can show that this equals  $ \mathcal{B}_\Lambda(z) B^*(y)$ by the arguments which go in parallel with the first part of the proof and rely on Lemma \ref{lemma_42}.

\noindent(2) By using the first assertion of the theorem and \eqref{blambdablambda},    we get
$$
B^*(z)\ts  B_\Lambda(z)
=B^*(z)\ts  \mathcal{B}_\Lambda(z)1
=\mathcal{B}_\Lambda(z)\ts B^*(z) 1
=\mathcal{B}_\Lambda(z)  1
=B_\Lambda(z),
$$
as required.
\end{prf}

Note that the series
 $B_\Lambda(z)$, defined by  \eqref{belambda_10}, can be regarded as an element of $\yhg[[z^{\pm 1}]]$. This  point of view is adopted in the following proposition.

\begin{pro}\label{Prop4455}
The coefficients of all series  $B_\Lambda(z)$ pairwise commute.
\end{pro}

\begin{prf}
Choose any two   series $B_\Lambda(x)$ and $B_\Gamma(y)$. By using \eqref{blambdablambda} and the first assertion of Theorem \ref{thm_41},  we get
\begin{align*}
&B_\Lambda(x)\ts B_\Gamma(y)
= B_\Lambda(x)\ts \mathcal{B}_\Gamma(y)  1
= \mathcal{B}_\Gamma(y)\ts B_\Lambda(x)   1
= \mathcal{B}_\Gamma(y)\ts \mathcal{B}_\Lambda(x)   1\\
= &\,\mathcal{B}_\Lambda(x)\ts\mathcal{B}_\Gamma(y)    1
= \mathcal{B}_\Lambda(x)\ts B_\Gamma(y)    1
=  B_\Gamma(y)\ts \mathcal{B}_\Lambda(x)   1
=  B_\Gamma(y)\ts B_\Lambda(x),
\end{align*}
as required.
\end{prf}

The next corollary follows by applying the maps from Proposition \ref{pro_21}, along with Remark \ref{tw_10}, to the commutative family established by Proposition \ref{Prop4455}.

\begin{kor}
The images of the coefficients of all series $B_\Lambda(z)$ under the map \eqref{map1_10} (resp. \eqref{map2_10}) form a commutative family in the orthogonal (resp.
symplectic) twisted $h$-Yangian. 
\end{kor}

At the end, we revisit the  Sklyanin determinant in the  algebra $\dyhoz$ framework; cf. \cite{MRS,Skly}. Introduce the $h$-permutation operator
$$
P^h=\sum_{i=1,\ldots ,N} e_{ii}\ot e_{ii}
+e^{h/2}\sum_{1\leqslant j<i\leqslant N} e_{ij}\ot e_{ji}
+e^{-h/2}\sum_{1\leqslant i<j\leqslant N} e_{ij}\ot e_{ji}.
$$
There exists a unique action of the symmetric group $\mathcal{S}_N$  on the tensor product space $\left(\CC^N\right)^{\ot N}$ so that for all transpositions $s_a=(a,a+1)$, $a=1,\ldots ,N-1$, we have
$
s_a \mapsto P_{s_a}^h\coloneqq P_{a\ts a+1}^h.
$
The action of the element $s\in \mathcal{S}_N$ with a reduced decomposition
$s=s_{a_1}\ldots s_{a_n}$ is given by $s\mapsto P_s\coloneqq P_{s_{a_1}}^h\ldots P_{s_{a_n}}^h$.  Moreover, the image of the antisymmetrizer under this action is given by
$$
A =\frac{1}{N!}\sum_{s\in \mathcal{S}_N}\mathop{sgn} s \cdot P_s^h.
$$

By \cite[Prop. 3.10]{KM}, the coefficients of the {\em quantum determinant}
$$
\mathop{qdet} T^+(u)=
\tr_{1,\ldots ,N} \, T_{[n]}^+ (u ,u-h\ldots ,u -(N-1)h)\vac \ts D^{\ot n}  A \in \Vcz [[u]]
$$
belong to the center of $\Vcz$. As with the critical level, by examining the image of the constant term of the quantum determinant $\mathop{qdet} T^+(u)$ under the map
$Y_{\Mcz }(\cdot , z)$, we find the series
$\mathop{Sdet} B(z)\in \Mcz [[z ]]$ and $\mathop{Sdet}\mathcal{B} (z)\in \dyhoz[[z^{\pm 1}]]$,
\begin{align*}
&\mathop{Sdet} B(z) = \tr_{1,\ldots ,N} \, B_{[N]}( z,ze^{-  h},\ldots ,ze^{-(N-1)h} ) \ts D^{\ot N}  A,\\
&
\mathop{Sdet}\mathcal{B} (z)=
\tr_{1,\ldots ,N} \,  B^+_{[N]}( z,ze^{-  h},\ldots ,ze^{-(N-1)h} )    B^*_{[N]}( z ,ze^{-  h },\ldots ,ze^{-(N-1)h } ) ^{-1}\ts D^{\ot N}  A.
\end{align*} 
We refer to both of the series as 
{\em Sklyanin determinants}. 
Their main property is given by the next proposition. It can be verified by the usual arguments which, in particular, rely on the properties of the antisymmetrizer and go in parallel  with \cite[Sect. 4.1]{MRS}.

\begin{pro}
\begin{enumerate}
\item  All coefficients of   $\mathop{Sdet}\mathcal{B} (z)$ belong to the center of   $\dyhoz$.
\item All coefficients of   $\mathop{Sdet} B(z)$ belong to the submodule of invariants $\mathfrak{z}(\Mcz)$.
\end{enumerate}
\end{pro}


\section*{Acknowledgement} 
L. B. is member of Gruppo Nazionale per le Strutture Algebriche, Geometriche e le loro Applicazioni  (GNSAGA) of the Istituto Nazionale di Alta Matematica (INdAM) and part of the project MMNLP (Mathematical Methods in Non Linear Physics) of INFN.
L. B. was partially supported by the project Representation Theory and Applications, Bando Ateneo 2023 of Sapienza University of Rome.
The research reported in this paper was in part carried out during the second author’s visit to the School of Mathematics and Statistics at the Central China Normal University in Wuhan. He is grateful to the School for warm hospitality.
S. K. is partially supported by the Croatian Science Foundation under the project IP-2025-02-4720  and by the project ``Implementation of cutting-edge research and its application as part of the Scientific Center of Excellence for Quantum and Complex Systems, and Representations of Lie Algebras'', Grant No. PK.1.1.10.0004, co-financed by the European Union through the European Regional Development Fund - Competitiveness and Cohesion Programme 2021--2027.
This research was supported by the European Union -- NextGenerationEU through the National Recovery and Resilience Plan 2021-2026. Institutional grant of University of Zagreb Faculty of Science Strengthening scientific production, international presence and social impact of mathematical research (IK IA 1.1.3. Impact4Math). 
J. Z. is supported in part by the National Natural Science Foundation of China (grant no. 12571026).

\end{document}